\documentclass[bj,numbers,noshowframe]{imsart}
\makeatletter\def\journal@name{}\makeatother
\usepackage{amsmath,amssymb,ascmac}
\usepackage{bm}
\usepackage{mathrsfs}

\allowdisplaybreaks[3]

\startlocaldefs
\numberwithin{equation}{section}
\theoremstyle{plain}
\newtheorem{tm}{Theorem}[section]
\newtheorem{lm}[tm]{Lemma}

\newtheorem{pr}[tm]{Proposition}
\theoremstyle{definition}
\newtheorem{re}[tm]{Remark}
\newtheorem{df}[tm]{Definition}

\newcommand{\R}{\mathbb{R}}
\newcommand{\N}{\mathbb{N}}

\newcommand{\dd}{\mathrm{d}}

\newcommand{\ve}{\varepsilon}
\newcommand{\dis}{\displaystyle}
\newcommand{\del}{\partial}

\newcommand{\Var}{\mathrm{Var}}

\newcommand{\nn}{\nonumber}

\newcommand{\E}{\mathbb{E}}

\endlocaldefs

\begin{document}

\begin{frontmatter}
\title{Asymptotic behaviors of fractional binomial distributions derived from the generalized binomial theorem}
\runtitle{Asymptotic behaviors of fractional binomial distributions}

\begin{aug}
\author[A]{\inits{M.}\fnms{Masanori}~\snm{Hino}\ead[label=e1]{hino@math.kyoto-u.ac.jp}}
\author[B]{\inits{R.}\fnms{Ryuya}~\snm{Namba}\ead[label=e2]{rnamba@cc.kyoto-su.ac.jp}}
\address[A]{Department of Mathematics, Graduate School of Science, Kyoto University, Kyoto 606-8502, Japan\printead[presep={,\ }]{e1}}

\address[B]{Department of Mathematics, Faculty of Science, Kyoto Sangyo University, Kyoto 603-8555, Japan\printead[presep={,\ }]{e2}}
\end{aug}

\begin{abstract}
A fractional binomial distribution, introduced by Hino and Namba (2024) via the generalized binomial theorem, is a fractional variant of the classical binomial distribution. Building upon previous work that established limit theorems, such as the weak law of large numbers and the central limit theorem, for fractional binomial distributions, this paper further investigates their asymptotic behaviors. Specifically, we derive the large and moderate deviation principles and a Berry--Esseen type estimate for these distributions.
\end{abstract}


\end{frontmatter}


\section{Introduction}\label{Sect:Introduction}
The \emph{fractional binomial distribution} $\mu_{\alpha,x}^{(n)}$, defined on the set $\{0,1,\dots,n\}$ with parameters $n\in\N$, $x\in[0,1]$, and $\alpha>0$, and derived from the generalized binomial theorem, was introduced in \cite{HN}. This distribution, inspired by the neo-classical inequality in rough path theory~\cite{Lyons}, serves as a natural extension of the classical binomial distribution. 
It is defined as 
\[
\sum_{j=0}^n \binom{\alpha n}{\alpha j}x^{\alpha j} (1-x)^{\alpha(n-j)}\delta_j
\]
divided by the normalizing constant, where $\delta_j$ denotes the Dirac measure at $j$.
In particular, $\mu_{1,x}^{(n)}$ coincides with the binomial distribution $\mathrm{Bin}(n,x)$, and its probability mass function depends smoothly on $x$ and $\alpha$. A key feature of this distribution is that its mean is asymptotically equal to that of $\mathrm{Bin}(n,x)$ as $n \to \infty$, but the variance of $\mu_{\alpha, x}^{(n)}$ is 
asymptotically equal to $1/\alpha$ times that of $\mathrm{Bin}(n,x)$. 
More precisely, as demonstrated in \cite[Corollary~2.8]{HN}:  
\begin{itemize}
\item the mean of $\mu_{\alpha, x}^{(n)}$ is $nx + O(e^{-\delta n})$, and 
\item the variance of $\mu_{\alpha, x}^{(n)}$ is $\frac{1}{\alpha}nx(1-x)+O(e^{-\delta n})$
\end{itemize}
as $n \to \infty$ for some $\delta>0$. 
Thus, the fractional binomial distributions offer flexible and novel count data distributions capable of accounting for both overdispersions and underdispersions. 

It is therefore a natural question to investigate the asymptotic behaviors of this new probability distribution. In \cite{HN}, several basic limit theorems were established, including the weak law of large numbers, the central limit theorem, and the law of small numbers.
To further deepen our understanding of the properties of these distributions, this paper investigates more advanced asymptotic behaviors. Specifically, this paper proves the large and moderate deviation principles for $\mu_{\alpha,x}^{(n)}$, along with a Berry--Esseen type estimate for $\mu_{\alpha,x}^{(n)}$, which quantitatively evaluates the error in the normal approximation. Interestingly, the rate functions in the large and moderate deviations differ from the classical ones by the factor of $\alpha$, whereas the Berry--Esseen type estimate is of the same form as the classical one. Given that $\mu_{\alpha,x}^{(n)}$ is not expected to be the law of the sum of independent random variables, the standard proof applicable to binomial distributions is not valid in our case. Instead, we utilize explicit expressions for the moment generating function and characteristic function of $\mu_{\alpha,x}^{(n)}$. Although these expressions are rather complicated, they can be quantitatively evaluated.
We anticipate that this clarification of asymptotic behaviors will lead potential applications in probabilistic and statistical models of count data.

The remainder of this paper is organized as follows. 
Section~\ref{Sect:results} provides the settings along with the details of the main results.
In Section~\ref{Sect:LDP-MDP}, we review the general theory of 
large deviation principles and the G\"artner--Ellis theorem.
Section~\ref{Sect:fractional-binomial-distribution} presents an explicit expression for the moment generating function of $\mu_{\alpha, x}^{(n)}$.
In Section~\ref{Sect:LDPs and MDPs}, we prove the large and moderate deviation principles.
Section~\ref{Sect:Berry-Esseen} provides the proof of a Berry--Esseen type estimate. 
Finally, Section~\ref{Sect:Remarks} offers some concluding remarks. 

\subsection*{Notation}
We use the following notation.
\begin{itemize}
    \item
        $\N=\{1, 2, 3, \dots\}$ denotes the set of natural numbers.
    \item
        $\mathrm{Re}(\omega)$ denotes the real part of 
        a complex number $\omega$. 
    \item 
        For $x, y \in \R$, we define 
        $x \vee y=\max\{x, y\}$ and $x \wedge y=\min\{x, y\}$. 
    \item 
        For $x \in \R$, $\lfloor x \rfloor$ denotes the greatest integer less than or equal to $x$, and $\lceil x \rceil$ denotes        the smallest integer greater than or 	equal to $x$. 
        \item 
        Let $f(n)$ and $g(n)$ be functions defined on $\N$. 
        Then, $f(n)=O(g(n))$ as $n \to \infty$
        signifies the big-$O$ notation. 
    \item 
        For a set $A \subset \R$, $A^\circ$ and $\del A$ denote the interior and the boundary of $A$, respectively. 
    \item 
        Positive constants, for which concrete expressions are not critical, are denoted by $C_{i}$ and $\delta_i$, representing the $i$-th positive constants. 
\end{itemize}



\section{Settings and statement of main results}\label{Sect:results}
We begin with some conventions.
We define 
    \[
    \binom{w}{z}:=\frac{\Gamma(w+1)}{\Gamma(z+1)\Gamma(w-z+1)}
    \]
for $w \in \mathbb{C} \setminus \{-k \mid k \in \N\}$
and $z \in \mathbb{C}$, 
where $\Gamma(\cdot)$ denotes the gamma function. 
If $\infty$ appears in the denominator of the right-hand side of the above, 
then $\binom{w}{z}$ is regarded as zero. 
Also, $0^0$ is defined as 1. 
The fractional binomial distribution is then defined as follows.

\begin{df}[{\cite[Definition~1.3]{HN}}]
\label{Def:fractional-binomial-distribution}
    Let $\alpha > 0$, $n \in \N$, and $x \in [0, 1]$. 
    A probability measure $\mu_{\alpha, x}^{(n)}$ on $\{0, 1, 2, \dots, n\}$, 
    called the $\alpha$-\emph{fractional binomial distribution}, is defined as 
    \[
        \mu_{\alpha, x}^{(n)}
        =\frac{\alpha}{Z_{\alpha, x}^{(n)}}
        \sum_{j=0}^n\binom{\alpha n}{\alpha j} x^{\alpha j}(1-x)^{\alpha(n-j)}\delta_j,
    \]
    where $Z_{\alpha, x}^{(n)}>0$ is the normalizing constant, that is,
    \[
    Z_{\alpha, x}^{(n)}=\alpha\sum_{j=0}^n \binom{\alpha n}{\alpha j} x^{\alpha j}(1-x)^{\alpha(n-j)}.
    \] 
\end{df}

\noindent
We note that $\mu_{1, x}^{(n)}$ coincides with the binomial distribution $\mathrm{Bin}(n,x)$, which means that  
$\mu_{\alpha, x}^{(n)}$ is regarded as 
a fractional variant of the binomial distribution. 

Let $S_{\alpha, x}^{(n)}$ denote a random variable 
whose law is $\mu_{\alpha, x}^{(n)}$. 
As shown in \cite[Corollary~2.8]{HN}, we have 
    \begin{equation}
        \mathbb{E}[S_{\alpha, x}^{(n)}] = nx + O(e^{-\delta n}) \quad\text{and}\quad
        \mathrm{Var}(S_{\alpha, x}^{(n)}) = \frac{1}{\alpha}nx(1-x)+O(e^{-\delta n})
        \label{Eq:mean-variance-fractional}
    \end{equation}
as $n \to \infty$, where $\delta>0$ can be chosen uniformly 
on every compact subset of $x \in (0, 1)$
and every compact subset of $\alpha \in (0, \infty)$. 
In our preceding paper \cite{HN}, 
we have proved some limit theorems for $\mu_{\alpha, x}^{(n)}$. 
We here present two of them, the weak law of large numbers (WLLN in short) and 
the central limit theorem (CLT in short).

\begin{tm}[WLLN for $\mu_{\alpha, x}^{(n)}$, {\cite[Theorem~3.1]{HN}}]
\label{Thm:WLLN}
Let $\alpha>0$, $n \in \N$, and $x \in [0, 1]$. 
Then, for every $\ve>0$ and a compact subset $K$ of $(0, \infty)$, we have  
    \begin{equation*}\label{Eq:WLLN}
        \lim_{n \to \infty}
        \sup_{\alpha \in K} \sup_{x \in [0, 1]} \mathbb{P}
        \left(\left|\frac{1}{n}S_{\alpha, x}^{(n)}-x\right| \ge \ve\right)=0. 
    \end{equation*}
\end{tm}

\begin{tm}[CLT for $\mu_{\alpha, x}^{(n)}$, {\cite[Theorem~3.2]{HN}}]
\label{Thm:CLT}
    Let $\alpha>0$, $n \in \N$, and $x \in (0, 1)$.
    Let $\widetilde{\mu}_{\alpha, x}^{(n)}$ denote the probability law of the random variable 
        \begin{equation}\label{Eq:normalized-mu}
            \widetilde{S}_{\alpha, x}^{(n)}
            :=\frac{S_{\alpha, x}^{(n)} - \mathbb{E}[S_{\alpha, x}^{(n)}]}{\sqrt{\mathrm{Var}(S_{\alpha, x}^{(n)})}}.
        \end{equation}
    Then, $\widetilde{\mu}_{\alpha, x}^{(n)}$
    converges weakly to the standard normal distribution $N(0, 1)$
    as $n \to \infty$. 
\end{tm}

The main results in this paper based on these theorems are summarized as follows. 
\smallskip

\noindent
{(I)}
\textit{Large and moderate deviation principles for $\mu_{\alpha, x}^{(n)}$.} 

\begin{tm}[LDP for $\mu_{\alpha, x}^{(n)}$]
\label{Thm:LDP-fractional}
    Let $\alpha>0$, $n \in \N$, and $x \in (0, 1)$.
    Then, the sequence of laws of 
    $S_{\alpha, x}^{(n)}/n$, $n \in \N$, satisfies the large deviation principle (LDP) 
    with speed $n$ and good rate function 
        \begin{equation}\label{Eq:rate-function-LDP}
            I_{\alpha, x}^{(\mathrm{L})}(z) 
            =\begin{cases}
            \displaystyle 
            \alpha \left\{z \log \frac{z}{x} 
            +  (1-z)\log\frac{1-z}{1-x}\right\} & \text{if }z \in [0, 1], \\
            +\infty & \text{otherwise}.
            \end{cases}
        \end{equation}
    Namely, the following two properties hold. 
    \begin{enumerate}
    \item[{\rm (1)}] For any closed set $F \subset \R$, we have 
        \[
        \limsup_{n \to \infty} \frac{1}{n}\log 
        \mathbb{P}\left(\frac{S_{\alpha, x}^{(n)}}{n} \in F\right)\le -\inf_{z \in F}I_{\alpha, x}^{({\rm L})}(z).
        \]
    \item[{\rm (2)}]
    For any open set $G \subset \R$, we have 
        \[
        \liminf_{n \to \infty} \frac{1}{n}\log \mathbb{P}\left(\frac{S_{\alpha, x}^{(n)}}{n} \in G\right) \ge -\inf_{z \in G}I_{\alpha, x}^{({\rm L})}(z).
        \]
    \end{enumerate}
\end{tm}

Next, let us focus on a moderate speed rate between $n$ of LLN-type
and $\sqrt{n}$ of CLT-type. 
An LDP corresponding to such a speed rate is occasionally called 
a \textit{moderate deviation principle} (MDP in short). 
It is known that MDPs can be obtained from neither LDPs nor CLTs.
Suppose that a positive sequence $\{c_n\}_{n=1}^\infty$ satisfies that
    \begin{equation}\label{Eq:moderate-rate}
         \lim_{n \to \infty}\frac{c_n}{n}=0 \quad \text{and} \quad
         \lim_{n \to \infty}\frac{c_n}{\sqrt{n}}=+\infty.
     \end{equation}
We show that the sequence of laws of 
the properly scaled random variables
    \[
        Y_{\alpha, x}^{(n)}:=\frac{S_{\alpha, x}^{(n)}-nx}{c_n}, \qquad 
        n \in \N,
    \]
satisfies an MDP as follows.

\begin{tm}[MDP for $\mu_{\alpha, x}^{(n)}$]
\label{Thm:MDP-fractional}
    Let $\alpha>0$, $n \in \N$, and $x \in (0, 1)$.
    Suppose that a positive sequence $\{c_n\}_{n=1}^\infty$
    satisfies \eqref{Eq:moderate-rate}.
    Then, the sequence of laws of 
    $Y_{\alpha, x}^{(n)}$, $n \in \N$, satisfies the MDP 
    with speed $c_n^2/n$ and  quadratic rate function 
        \[
            I_{\alpha, x}^{(\mathrm{M})}(z) 
            =\frac{\alpha}{2x(1-x)}z^2, \qquad z \in \R.
        \]
    Namely, the following two properties hold.
    \begin{enumerate}
    \item[{\rm (1)}] For any closed set $F \subset \R$, we have 
        \[
        \limsup_{n \to \infty} \frac{n}{c_n^2}\log 
        \mathbb{P}\left(Y_{\alpha, x}^{(n)} \in F\right)\le -\inf_{z \in F}I_{\alpha, x}^{({\rm M})}(z).
        \]
    \item[{\rm (2)}]
    For any open set $G \subset \R$, we have 
        \[
        \liminf_{n \to \infty} \frac{n}{c_n^2}\log \mathbb{P}\left(Y_{\alpha, x}^{(n)} \in G\right) \ge -\inf_{z \in G}I_{\alpha, x}^{({\rm M})}(z).
        \]
    \end{enumerate}
\end{tm}
Note that the rate functions in Theorems~\ref{Thm:LDP-fractional} and \ref{Thm:MDP-fractional} are equal to those of the case of classical binomial distributions multiplied by $\alpha$.

\smallskip

\noindent
{(II)} 
\textit{Berry--Esseen type estimate for $\mu_{\alpha, x}^{(n)}$.} 
Since we proved the central limit theorem for $\mu_{\alpha, x}^{(n)}$ in \cite[Theorem~3.2]{HN}, 
it is natural to ask for a quantitative error estimate of the CLT
as a subsequent question. 
Recall \eqref{Eq:normalized-mu}.
Let $\widetilde{F}_{\alpha, x}^{(n)}(z)=\mathbb{P}(\widetilde{S}_{\alpha, x}^{(n)} \le z)$ $(z \in \R)$ be the distribution function of $\widetilde{S}_{\alpha, x}^{(n)}$. 
Moreover, let $\Phi$ denote the distribution function of the standard normal distribution $N(0,1)$, that is, 
    \[
        \Phi(z)=\int_{-\infty}^z \frac{1}{\sqrt{2\pi}}
        e^{-u^2/2} \, \dd u, \qquad z \in \R.
    \]
A quantitative estimate of the difference between $\widetilde{F}_{\alpha, x}^{(n)}(z)$ and $\Phi(z)$ is known as a \emph{Berry--Esseen type estimate},
which is stated as follows.

\begin{tm}[Berry--Esseen type estimate for $\mu_{\alpha, x}^{(n)}$]
\label{Thm:Berry-Esseen}
    There exists some positive constant $C$
    depending on $\alpha>0$ and $x \in (0, 1)$ such that 
    \begin{equation}\label{Eq:BE-estimate}
        \sup_{z \in \R}|\widetilde{F}_{\alpha, x}^{(n)}(z) - \Phi(z)|  \le \frac{C}{\sqrt{n}}, \qquad  n \in \N.   
    \end{equation}
\end{tm}

\section{LDPs and the G\"artner--Ellis theorem}
\label{Sect:LDP-MDP}

In this section, we review the general definition of 
LDPs and the G\"artner--Ellis theorem for the proof of Theorems~\ref{Thm:LDP-fractional} and \ref{Thm:MDP-fractional}.
We refer to e.g., \cite{DZ} for details on 
the general theory of LDPs.

A lower semi-continuous map $I \colon \R \to [0, \infty]$ 
is called a \emph{rate function}. 
If a rate function $I$ has the property that each level set
$\{x \in \R \mid I(x) \le a\}$ of $I$ is a compact subset of $\R$
for all $a \ge 0$, then $I$ is said to be \emph{good}. 
The set $\mathcal{D}_I=\{x \in \R \mid I(x)<\infty\}$ is 
called the \emph{effective domain} of $I$. 
Let $\mathcal{B}(\R)$ denote the Borel $\sigma$-field on $\R$.

\begin{df}[Large deviation principle, cf.~{\cite[p.\,5]{DZ}}]\label{Def:LDP-general-definition}
    Let $\{a_n\}_{n=1}^\infty$ be a positive sequence satisfying 
    $a_n \to \infty$ as $n \to \infty$. 
    A sequence $\{\mu_n\}_{n=1}^\infty$ of probability measures on $(\R, \mathcal{B}(\R))$ satisfies the large deviation principle 
    with speed $a_n$ and
    rate function $I$ if 
    the following two properties hold. 
    \begin{enumerate}
    \item[{\rm (1)}]
    For any closed set $F \subset \R$, we have 
        \[
        \limsup_{n \to \infty} \frac{1}{a_n}\log \mu_n(F) \le -\inf_{z \in F}I(z).
        \]
    \item[{\rm (2)}]
    For any open set $G \subset \R$, we have 
        \[
        \liminf_{n \to \infty} \frac{1}{a_n}\log \mu_n(F) \ge -\inf_{z \in G}I(z).
        \]
\end{enumerate} 
\end{df}

\noindent
The most typical choice of the positive sequence 
$\{a_n\}_{n=1}^\infty$ is $a_n=n$, $n \in \N$.

Let $\{Y_n\}_{n=1}^\infty$ be a sequence of real random variables. 
For $n \in \N$, let  
    \[
    \Lambda_n(\xi):=\log \mathbb{E}[e^{\xi Y_n}], 
    \qquad \xi \in \R,
    \]
be the logarithmic moment generating function of $Y_n$.
We now put the following two assumptions on $\Lambda_n$. 

\begin{enumerate}
    \item[{(A1)}] 
    There exists the limit 
        \[
        \Lambda(\xi):=\lim_{n \to \infty}\frac{1}{a_n}\Lambda_n(a_n\xi) \, (\in [-\infty, \infty]), \qquad \xi \in \R. 
        \]
    \item[{(A2)}] 
    It holds that $0 \in (\mathcal{D}_\Lambda)^\circ$, where 
    $\mathcal{D}_\Lambda$ is the effective domain of $\Lambda$. 
\end{enumerate}

Let $\Lambda^*$ be the \emph{Fenchel--Legendre transform} of $\Lambda_n$, that is, 
    \[
    \Lambda^*(z):=\sup_{\xi \in \R}\left(\xi z - \Lambda(\xi)\right), \qquad z \in \R. 
    \]
We note that $\Lambda$ is convex and $\Lambda^*$ is a good rate function 
provided that both {\bf (A1)} and {\bf (A2)} hold. 
The following plays a crucial role in the present paper. 

\begin{tm}[The G\"artner--Ellis theorem, cf.~{\cite[Theorem~2.3.6]{DZ}}]
\label{Thm:Gartner-Ellis}
    Assume both {\rm (A1)} and {\rm (A2)}. 
    Then, the following hold. 

    \begin{enumerate}
        \item[{\rm (1)}] For any closed set $F \subset \R$, we have 
            \[
            \limsup_{n \to \infty} \frac{1}{a_n} \log \mathbb{P}(Y_n \in F) \le -\inf_{z \in F}\Lambda^*(z). 
            \]
        \item[{\rm (2)}] Assume further the following: 
            \begin{itemize}
                \item[(a)]
                    $\Lambda$ is lower semi-continuous on $\R$,
                \item[(b)] 
                    $\Lambda$ is differentiable on $(\mathcal{D}_\Lambda)^\circ$, 
                \item[(c)]
                    Either $\mathcal{D}_\Lambda=\R$ or 
                    $\Lambda$ is steep at $\partial \mathcal{D}_\Lambda$, that is, 
                    $\dis \lim_{\xi \to \del \mathcal{D}_\Lambda, \, \xi \in \mathcal{D}_\Lambda}|\Lambda'(\xi)|<\infty $. 
            \end{itemize}
        Then, for any open set $G \subset \R$, we have 
            \[
            \liminf_{n \to \infty} \frac{1}{a_n} \log \mathbb{P}(Y_n \in G) \ge -\inf_{z \in G}\Lambda^*(z),
            \]
        which means that the sequence $\{\mu_n=\mathbb{P} \circ Y_n^{-1}\}_{n=1}^\infty$ satisfies 
        the LDP with speed $a_n$ and 
         good rate function $\Lambda^*(z)$. 
    \end{enumerate}
\end{tm}

\section{Moment generating function of the fractional binomial distribution}
\label{Sect:fractional-binomial-distribution}
In this section, we give an alternative representation of the normalizing constant $Z_{\alpha, x}^{(n)}$ together with the explicit expression of the moment generating function of $\mu_{\alpha, x}^{(n)}$.

For $z \in \mathbb{C} \setminus \{x \in \R \mid x \le 0\}$ 
and $\gamma \in \R$,
the power $z^\gamma$ is defined as $\exp(\gamma \,\mathrm{Log}\,z)$, 
where $\mathrm{Log}\,z$ is the principal value of $\log z$ 
so that $\mathrm{Log}\,1=0$. 
For $\alpha>0$, we put 
\[
    K_\alpha:=\{e^{i\theta} \in \mathbb{C} 
    \mid -\pi < \theta \le \pi, \, e^{i \theta \alpha}=1\}. 
\]
We note that $K_\alpha=\{1\}$ when $\alpha \in (0, 2)$ and 
that $-1 \in K_\alpha$ if and only if $\alpha/2 \in \N$. 

A key identity is stated as follows.

\begin{tm}[Generalized binomial theorem, see {\cite[Theorem~3.2]{HH}}
and {\cite[Lemma~2.13]{HN}}]
\label{Thm:generalization-binomial-theorem-HH}
    For $\alpha>0$, $n \in \N$, and $\lambda >0$, 
    we have 
    \begin{align}\label{Eq:Generalized-binomial-theorem-HH}
        \alpha\sum_{j=0}^n \binom{\alpha n}{\alpha j} \lambda^{\alpha j}
        &= \sum_{\omega \in K_\alpha}(1+\lambda \omega)^{\alpha n}
            -\frac{\alpha \lambda^\alpha \sin \alpha \pi}{\pi}
            \int_0^1 t^{\alpha-1}(1-t)^{\alpha n} \nonumber \\*
        &\qquad\times \left\{
        \frac{1}{|t^\alpha-\lambda^\alpha e^{-i\alpha \pi}|^2}
        +\frac{\lambda^{\alpha n}}{|e^{-i\alpha \pi}-(\lambda t)^\alpha|^2}
        \right\} \, \dd t. 
    \end{align}
    Here, the second term on the right-hand side is regarded as zero if $\alpha/2 \in \N$. 
\end{tm}
In particular, by putting $\lambda=x/(1-x)$, $x \in (0, 1)$, in \eqref{Eq:Generalized-binomial-theorem-HH}
and multiplying both sides by $(1-x)^n$, we obtain 
\begin{align*}
    Z_{\alpha, x}^{(n)}
    &=\alpha\sum_{j=0}^n \binom{\alpha n}{\alpha j} x^{\alpha j}(1-x)^{\alpha(n-j)} \nn \\
    &= 1+\sum_{\omega \in K_\alpha \setminus \{1\}}(1-x+x\omega)^{\alpha n}
    -\frac{\alpha x^\alpha (1-x)^\alpha \sin \alpha \pi}{\pi}
    \int_0^1 t^{\alpha-1}(1-t)^{\alpha n} \nonumber \nn \\*
    &\qquad\times \left[
    \frac{(1-x)^{\alpha n}}{|\{t(1-x)\}^\alpha-x^\alpha e^{-i\alpha \pi}|^2}
    +\frac{x^{\alpha n}}{|(1-x)^\alpha e^{-i\alpha \pi}-(tx)^\alpha|^2}
    \right] \, \dd t 
\end{align*}
for $\alpha>0$, $n \in \N$, and $x \in (0, 1)$. 
We observe that 
$Z_{\alpha, x}^{(n)}$ is close to 1 when $n$ is sufficiently large.
More precisely, it holds that 
\begin{equation}\label{Eq:Z}
Z_{\alpha, x}^{(n)}=1+O(e^{-\delta n})\quad\text{as }n \to \infty,
\end{equation}
where $\delta>0$ is chosen uniformly on every compact subset of $x \in (0, 1)$
and on every compact subset of $\alpha \in (0, \infty)$. 
It is worth mentioning that the terminology ``generalized binomial theorem''
is often used for indicating the power series expansion 
     \[
	(1+x)^\alpha = \sum_{n=0}^{\infty} \binom{\alpha}{n} x^{n},
	\qquad |x|<1,\ \alpha \in \mathbb{R}.
      \]
Here, Theorem~\ref{Thm:generalization-binomial-theorem-HH}
is another kind of generalization of the classical binomial theorem. 

Let $\alpha>0$, $n \in \N$, and $x \in (0, 1)$. 
Let us recall that $S_{\alpha, x}^{(n)}$ is a random variable 
whose law is $\mu_{\alpha, x}^{(n)}$. 
The moment generating function $M_{\alpha, x}^{(n)}$ of $\mu_{\alpha, x}^{(n)}$ is defined as
    \begin{equation}\label{Eq:MGF-fractional}
        M_{\alpha, x}^{(n)}(\xi)
        := \mathbb{E}\left[e^{\xi S_{\alpha, x}^{(n)}}\right]=
        \sum_{j=0}^n e^{\xi j}\mu_{\alpha, x}^{(n)}(\{j\}),
        \qquad \xi \in \R.
    \end{equation}
In the manner similar to the proof of \cite[Theorem~2.15]{HN}, 
we have the following.

\begin{tm}[Explicit expression of $M_{\alpha, x}^{(n)}$]\label{Thm:MGF-fractional-binomial}
    Let $\alpha>0$, $n \in \N$, $x \in (0, 1)$,
    and $\xi \in \R$.
    Then, we have 
    \begin{align}
        M_{\alpha, x}^{(n)}(\xi)
        &=\frac{1}{Z_{\alpha, x}^{(n)}}\left(1-x+xe^{\xi/\alpha }\right)^{\alpha n} 
        +\sum_{\omega \in K_\alpha \setminus \{1\}}
        \frac{1}{Z_{\alpha, x}^{(n)}}
        \left(1-x+x\omega e^{\xi/\alpha }\right)^{\alpha n} \nn \\
        &\quad-\frac{1}{Z_{\alpha, x}^{(n)}}
        \frac{\alpha e^{\xi}\sin \alpha \pi}{\pi}\int_0^1 t^{\alpha-1}(1-t)^{\alpha n} \nonumber \\
        &\quad\times \Biggl[
        \frac{x^\alpha (1-x)^{\alpha(n+1)}}{|\{t(1-x)\}^\alpha-x^\alpha e^{\xi-i\alpha \pi}|^2} 
        +\frac{x^{\alpha (n+1)}(1-x)^{\alpha}e^{\xi n}}
        {|(1-x)^\alpha -(tx)^\alpha e^{\xi+i\alpha \pi}|^2}
        \Biggr] \, \dd t. 
        \label{Eq:CF-fractional-binomial-1}
        \end{align}
    Here, if $\alpha \in \N$, then the last term of \eqref{Eq:CF-fractional-binomial-1} 
    is regarded as $0$. 
\end{tm}

\begin{proof}
For $\xi \in \R$, we have 
    \begin{align*}
        M_{\alpha, x}^{(n)}(\xi)
        &=\sum_{j=0}^n e^{\xi j} 
        \frac{\alpha}{Z_{\alpha, x}^{(n)}}
        \binom{\alpha n}{\alpha j}x^{\alpha j}(1-x)^{\alpha(n-j)} \\
        &=\frac{\alpha}{Z_{\alpha, x}^{(n)}}
        \sum_{j=0}^n \binom{\alpha n}{\alpha j}
        \left(\frac{xe^{\xi/\alpha}}{1-x}\right)^{\alpha j}(1-x)^{\alpha n}. 
    \end{align*}
We put $\lambda=xe^{\xi/\alpha}/(1-x)>0$. 
We apply Theorem~\ref{Thm:generalization-binomial-theorem-HH} to obtain 
    \begin{align*}
        M_{\alpha, x}^{(n)}(\xi) 
        &=\frac{1}{Z_{\alpha, x}^{(n)}}(1-x)^{\alpha n}  
        \alpha \sum_{j=0}^n \binom{\alpha n}{\alpha j} \lambda^{\alpha j} \\
        &=\frac{1}{Z_{\alpha, x}^{(n)}}(1-x)^{\alpha n} \left[ 
        \sum_{\omega \in K_\alpha}(1+\lambda\omega)^{\alpha n}
        -\frac{\alpha \lambda^\alpha \sin \alpha \pi}{\pi}
        \int_0^1 t^{\alpha-1}(1-t)^{\alpha n}\right. \nonumber \\
        &\quad\times \left.\left\{
        \frac{1}{|t^\alpha-\lambda^\alpha e^{-i\alpha \pi}|^2}
        +\frac{\lambda^{\alpha n}}{|e^{-i\alpha \pi}-(\lambda t)^\alpha|^2}
        \right\} \, \dd t \right] \\
        &=\frac{1}{Z_{\alpha, x}^{(n)}}\left(1-x+xe^{\xi/\alpha }\right)^{\alpha n} 
        +\sum_{\omega \in K_\alpha \setminus \{1\}}
        \frac{1}{Z_{\alpha, x}^{(n)}}
        \left(1-x+x\omega e^{\xi/\alpha}\right)^{\alpha n} \nn \\
        &\quad-\frac{1}{Z_{\alpha, x}^{(n)}}
        \frac{\alpha e^{\xi}\sin \alpha \pi}{\pi}\int_0^1 t^{\alpha-1}(1-t)^{\alpha n} \nonumber \\
        &\quad\times \left[
        \frac{x^\alpha (1-x)^{\alpha(n+1)}}{|\{t(1-x)\}^\alpha-x^\alpha e^{\xi-i\alpha \pi}|^2} 
        +\frac{x^{\alpha (n+1)}(1-x)^{\alpha}e^{\xi n}}
        {|(1-x)^\alpha -(tx)^\alpha e^{\xi+i\alpha \pi}|^2}
        \right] \, \dd t,
    \end{align*}
where the last term is regarded as $0$ if $n \in \N$. 
\end{proof}

\begin{re}\normalfont
    In \cite[Theorem~2.15]{HN}, we have obtained an explicit expression
    of the characteristic function 
    $\varphi_{\alpha, x}^{(n)}(\xi)=\mathbb{E}[e^{i\xi S_{\alpha, x}^{(n)}}]$, $\xi \in \R$, of $\mu_{\alpha, x}^{(n)}$
    through the generalized binomial theorem (see \eqref{Eq:fractional-CF-explicit} below). 
    Since the expression \eqref{Eq:fractional-CF-explicit} of $\varphi_{\alpha, x}^{(n)}(\xi)$ has apparent singularities,
    it is valid only in a neighborhood of $\xi=0$. 
    In contrast of this, \eqref{Eq:CF-fractional-binomial-1} 
    has no restrictions on $\xi \in \R$. 
\end{re}

\section{LDPs and MDPs for fractional binomial distributions}
\label{Sect:LDPs and MDPs}

\subsection{LDPs for fractional binomial distributions}
\label{Subsect:LDPs}

Let $\alpha>0$, $n \in \N$, and $x \in (0, 1)$.
The purpose of this section is to prove Theorem~\ref{Thm:LDP-fractional} by applying the G\"artner--Ellis theorem (Theorem~\ref{Thm:Gartner-Ellis}). 

Before demonstrating Theorem~\ref{Thm:LDP-fractional}, 
we give the following lemma. 

\begin{lm}\label{Lem:log-inequality}
    Let $N \ge2$ and $\{A_{n}^{(k)}\}_{n=1}^\infty$, $k=1, 2, \dots, N$, 
    be complex sequences such that $A_n^{(1)}$ is a positive real number for $n \in \N$. 
    Suppose that a positive sequence $\{a_n\}_{n=1}^\infty$ satisfies 
    $a_n \to \infty$ and there exists  
    \[
        \gamma := \lim_{n \to \infty}\frac{1}{a_n}\log A_{n}^{(1)} \in \R.
    \]
    If it holds that 
    \[
        \limsup_{n \to \infty}\frac{1}{a_n}\log |A_{n}^{(k)}| < \gamma,
        \qquad k=2,3, \dots, N,
    \]
    then we have 
    \[
        \lim_{n \to \infty} \frac{1}{a_n}\mathrm{Log} \left( \sum_{k=1}^N A_{n}^{(k)}\right)=\gamma.
    \]
\end{lm}

\begin{proof}
    Let 
        \[
        b_n=\frac{1}{A_n^{(1)}}\sum_{k=2}^N A_n^{(k)}. 
        \]
    We take $\gamma'$ and $\gamma''$ such that 
        \[
        \max_{k=2,3, \dots, N}
        \limsup_{n \to \infty}\frac{1}{a_n}\log |A_n^{(k)}|
        <\gamma' < \gamma'' < \gamma.
        \]
    Then, for sufficiently large $n$, it holds that 
    $A_n^{(1)} \ge e^{\gamma'' a_n}$ and 
    $|A_n^{(k)}| \le e^{\gamma'a_n}$
    for $k=2,3, \dots, N$.
    This implies that 
    $|b_n| \le Ne^{(\gamma'-\gamma'')a_n} \to 0$
    as $n \to \infty$. 
    Therefore, for sufficiently large $n$, we have 
    \[
      \frac{1}{a_n}\mathrm{Log}\left(\sum_{k=1}^N A_n^{(k)}\right)
            =\frac{1}{a_n}\log A_n^{(1)}
            +\frac{1}{a_n}\mathrm{Log}\,(1+b_n) \to \gamma
            \quad (n \to \infty).\qedhere
    \]
    \end{proof}

We are now ready for the proof of Theorem~\ref{Thm:LDP-fractional}. 

\begin{proof}[Proof of Theorem~\ref{Thm:LDP-fractional}]

Let $\Lambda_{\alpha, x}^{(n)}$ be the logarithmic moment generating function 
of $S_{\alpha, x}^{(n)}/n$, that is, 
    \[
    \Lambda_{\alpha, x}^{(n)}(\xi)
    :=\log M_{\alpha, x}^{(n)}\left(\frac{\xi}{n}\right), \qquad \xi \in \R,
    \]
where $M_{\alpha, x}^{(n)}$ is the moment generating function given by \eqref{Eq:MGF-fractional}. 
Suppose that $\alpha \notin \N$. 
Then, it follows from Theorem~\ref{Thm:MGF-fractional-binomial} that 
    \begin{align}
        \frac{1}{n}\Lambda_{\alpha, x}^{(n)}(n\xi)
        &= \frac{1}{n} \log \Biggl[
        \frac{1}{Z_{\alpha, x}^{(n)}}\left(1-x+xe^{\xi/\alpha}\right)^{\alpha n} 
        +\sum_{\omega \in K_\alpha \setminus \{1\}} \frac{1}{Z_{\alpha, x}^{(n)}}
        \left(1-x+x\omega e^{\xi/\alpha}\right)^{\alpha n} \nn \\
        &\quad-\frac{1}{Z_{\alpha, x}^{(n)}}
        \frac{\alpha e^{\xi} x^\alpha (1-x)^{\alpha(n+1)} \sin \alpha \pi}{\pi}\int_0^1 \frac{t^{\alpha-1}(1-t)^{\alpha n}}{|\{t(1-x)\}^\alpha-x^\alpha e^{\xi-i\alpha \pi}|^2} \, \dd t \nonumber \\
        &\quad-\frac{1}{Z_{\alpha, x}^{(n)}}
        \frac{\alpha e^{\xi} x^{\alpha(n+1)} (1-x)^{\alpha} \sin \alpha \pi}{\pi}\int_0^1 \frac{t^{\alpha-1}(1-t)^{\alpha n}e^{\xi n}}
        {|(1-x)^\alpha -(tx)^\alpha e^{\xi+i\alpha \pi}|^2} \, \dd t \Biggr] \nn \\
        &=: \frac{1}{n}\log \left(A_{n}^{(1)}+\sum_{\omega \in K_\alpha \setminus \{1\}}A_{n}^{(2)}(\omega)+A_{n}^{(3)}+A_{n}^{(4)} \right). \label{Eq:LDP-for-GE-0}
    \end{align}
From \eqref{Eq:Z}, we have 
    \begin{align}
        \frac{1}{n}\log A_n^{(1)}
        &= -\frac{1}{n}\log Z_{\alpha, x}^{(n)}
        +\alpha \log \left(1-x+x e^{\xi/\alpha}\right) \nn \\
        &\to \alpha \log \left(1-x+x e^{\xi/\alpha}\right)
        \label{Eq:LDP-for-GE-1}
    \end{align}
as $n \to \infty$. On the other hand, we observe that 
    \begin{align*}
        \left|1-x+x \omega e^{\xi/\alpha}\right|^2
        &=(1-x)^2+2x(1-x)\mathrm{Re}(\omega)e^{\xi/\alpha}+x^2e^{2\xi/\alpha} \nn \\
        &< (1-x)^2+2x(1-x)e^{\xi/\alpha}+x^2e^{2\xi/\alpha} 
        =\left(1-x+x e^{\xi/\alpha}\right)^2.
    \end{align*}
for $\omega \in K_\alpha \setminus \{1\}$. Hence, we have
    \begin{equation}\label{Eq:LDP-for-GE-2}
        \limsup_{n \to \infty}\frac{1}{n}\log |A_n^{(2)}(\omega)| 
        = \alpha \log \left|1-x+x\omega e^{\xi/\alpha}\right|
        < \alpha \log \left(1-x+x e^{\xi/\alpha}\right)
    \end{equation}
for $\omega \in K_\alpha \setminus \{1\}$.
Furthermore, from the inequalities 
    \begin{align}
        \left|\{t(1-x)\}^\alpha-x^\alpha e^{\xi-i\alpha \pi}\right|^2
        &\ge x^{2\alpha}e^{2\xi}\sin^2\alpha \pi >0
        \label{Eq:lower-bound-1} 
        \intertext{and}
        \left|(1-x)^\alpha-(tx)^\alpha e^{\xi+i\alpha \pi}\right|^2
        &=|(1-x)^\alpha e^{-i\alpha \pi}-(tx)^\alpha e^\xi|^2 \nn \\
        &\ge (1-x)^{2\alpha}\sin^2\alpha \pi>0,
        \label{Eq:lower-bound-2}
    \end{align}
it follows that 
    \begin{align}
        \frac{1}{n}\log |A_n^{(3)}|
        &\le \frac{1}{n}\log \left\{\frac{1}{Z_{\alpha, x}^{(n)}}
        \frac{\alpha e^{-\xi} x^{-\alpha} (1-x)^{\alpha(n+1)} }{\pi |\sin \alpha \pi|}
        \int_0^1 t^{\alpha-1}(1-t)^{\alpha n} \, \dd t\right\} \nn \\
        &\le \frac{1}{n}\log \left\{\frac{1}{Z_{\alpha, x}^{(n)}}
        \frac{\alpha e^{-\xi} x^{-\alpha} (1-x)^{\alpha(n+1)} }{\pi |\sin \alpha \pi|}
        \int_0^1 t^{\alpha-1}  \, \dd t\right\} \nn \\
        &= \frac{1}{n}\log \left\{\frac{1}{Z_{\alpha, x}^{(n)}}
        \frac{ e^{-\xi} x^{-\alpha} (1-x)^{\alpha(n+1)} }{\pi |\sin \alpha \pi|}\right\} \nn \\
        &= -\frac{1}{n}\log Z_{\alpha, x}^{(n)}
        + \frac{1}{n}\log \frac{e^{-\xi} x^{-\alpha}}{\pi|\sin \alpha \pi|} + \alpha \left(1+\frac{1}{n}\right)\log (1-x) \nn
    \end{align}
and 
    \begin{align}
        \frac{1}{n}\log |A_n^{(4)}|
        &\le \frac{1}{n}\log \left\{\frac{1}{Z_{\alpha, x}^{(n)}}
        \frac{\alpha e^{\xi(n+1)} x^{\alpha(n+1)} (1-x)^{-\alpha} }{\pi |\sin \alpha \pi|}
        \int_0^1 t^{\alpha-1}(1-t)^{\alpha n} \, \dd t\right\} \nn \\
        &\le \frac{1}{n}\log \left\{\frac{1}{Z_{\alpha, x}^{(n)}}
        \frac{\alpha e^{\xi(n+1)} x^{\alpha(n+1)} (1-x)^{-\alpha} }{\pi |\sin \alpha \pi|}
        \int_0^1 t^{\alpha-1}  \, \dd t\right\} \nn \\
        &= \frac{1}{n}\log \left\{\frac{1}{Z_{\alpha, x}^{(n)}}
        \frac{ e^{\xi(n+1)} x^{\alpha(n+1)} (1-x)^{-\alpha} }{\pi |\sin \alpha \pi|}\right\} \nn \\
        &= -\frac{1}{n}\log Z_{\alpha, x}^{(n)}
        + \frac{1}{n}\log \frac{(1-x)^{-\alpha}}{\pi|\sin \alpha \pi|}+\xi\left(1+\frac{1}{n}\right)+ \alpha \left(1+\frac{1}{n}\right)\log x. \nn
    \end{align}
Thus, we have 
    \begin{align}
        \limsup_{n \to \infty}\frac{1}{n}\log |A_n^{(3)}|
        &\le \alpha \log (1-x) < \alpha \log\left(1-x+xe^{\xi/\alpha}\right), 
        \label{Eq:LDP-for-GE-3}\\
        \limsup_{n \to \infty}\frac{1}{n}\log |A_n^{(4)}|
        &\le \xi+\alpha \log x = \alpha \log \left(x e^{\xi/\alpha}\right) < \alpha \log(1-x+xe^{\xi/\alpha}). 
        \label{Eq:LDP-for-GE-4}
    \end{align}
From \eqref{Eq:LDP-for-GE-0}, \eqref{Eq:LDP-for-GE-1}, \eqref{Eq:LDP-for-GE-2}, 
\eqref{Eq:LDP-for-GE-3}, and \eqref{Eq:LDP-for-GE-4}, 
Lemma~\ref{Lem:log-inequality} allows us to have the existence of 
the limit
    \[
    \Lambda_{\alpha, x}(\xi):=
    \lim_{n \to \infty}\frac{1}{n}\Lambda_{\alpha, x}^{(n)}(n\xi)
    =\alpha \log (1-x+xe^{\xi/\alpha}), \qquad \xi \in \R, 
    \]
with $\mathcal{D}_{\Lambda_{\alpha, x}}=\R$, which turns out to satisfy 
(a), (b), and (c) in Theorem~\ref{Thm:Gartner-Ellis}-(2). 
Theorem~\ref{Thm:Gartner-Ellis} implies that 
the sequence of laws of $\{S_{\alpha, x}^{(n)}/n\}_{n=1}^\infty$
does satisfy the LDP with speed $n$ and 
good rate function $I_{\alpha, x}^{(\mathrm{L})}$ defined by \eqref{Eq:rate-function-LDP}, 
since we can easily confirm that the Fenchel--Legendre transform of $\Lambda_{\alpha, x}$ is $I_{\alpha, x}^{(\mathrm{L})}$.

The case where $\alpha \in \N$ is more simply proved.
\end{proof}

\subsection{MDPs for fractional binomial distributions}
\label{Subsect:MDPs}

In this section, we give a proof of Theorem~\ref{Thm:MDP-fractional}.

\begin{proof}[Proof of Theorem~\ref{Thm:MDP-fractional}]
Let $\alpha>0$, $n \in \N$, and $x \in (0, 1)$.
Suppose that a positive sequence $\{c_n\}_{n=1}^\infty$
satisfies \eqref{Eq:moderate-rate}. 
    Let $\widetilde{\Lambda}_{\alpha, x}^{(n)}$ be the logarithmic moment generating function 
    of $Y_{\alpha, x}^{(n)}$, that is, 
        \[
            \widetilde{\Lambda}_{\alpha, x}^{(n)}(\xi)
            :=\log \mathbb{E}\left[e^{\xi Y_{\alpha, x}^{(n)}}\right]
            =\log \left(e^{-nx\xi/c_n}M_{\alpha, x}^{(n)}\left(\frac{\xi}{c_n}\right)\right), \qquad \xi \in \R.
        \]
    Suppose that $\alpha \notin \N$. 
    Then, Theorem~\ref{Thm:MGF-fractional-binomial} implies
        \begin{align}
            &\frac{n}{c_n^2}\widetilde{\Lambda}_{\alpha, x}^{(n)}
            \left(\frac{c_n^2}{n}\xi\right) \nn \\
            &=\frac{n}{c_n^2} \log \Biggl[ 
            \frac{1}{Z_{\alpha, x}^{(n)}} e^{-c_n x\xi}
            \left\{1-x+xe^{c_n\xi/(\alpha n)}\right\}^{\alpha n} \nn \\
            &\qquad+\sum_{\omega \in K_\alpha \setminus \{1\}} \frac{1}{Z_{\alpha, x}^{(n)}}e^{-c_n x\xi}
            \left\{1-x+x\omega e^{c_n\xi/(\alpha n)}\right\}^{\alpha n} \nn \\
            &\qquad-\frac{1}{Z_{\alpha, x}^{(n)}}
            \frac{\alpha e^{-c_n\xi(x-1/n)} x^\alpha (1-x)^{\alpha(n+1)} \sin \alpha \pi}{\pi}\int_0^1 \frac{t^{\alpha-1}(1-t)^{\alpha n}}{|\{t(1-x)\}^\alpha-x^\alpha e^{(c_n\xi/n)-i\alpha \pi}|^2} \, \dd t \nonumber \\
            &\qquad-\frac{1}{Z_{\alpha, x}^{(n)}}
            \frac{\alpha e^{-c_n\xi(x-1/n)} x^{\alpha(n+1)} (1-x)^{\alpha} \sin \alpha \pi}{\pi}\int_0^1 \frac{t^{\alpha-1}(1-t)^{\alpha n}e^{c_n\xi }}
            {|(1-x)^\alpha -(tx)^\alpha e^{(c_n\xi/n)+i\alpha \pi}|^2} \, \dd t \Biggr] \nn \\
            &=: \frac{n}{c_n^2}\log \left(\widetilde{A}_{n}^{(1)}+\sum_{\omega \in K_\alpha \setminus \{1\}}\widetilde{A}_{n}^{(2)}(\omega)+\widetilde{A}_{n}^{(3)}+\widetilde{A}_{n}^{(4)} \right). \label{Eq:MDP-for-GE-0}
        \end{align}
        Then, it holds that
        \begin{align}
            \frac{n}{c_n^2}\log \widetilde{A}_n^{(1)}
            &= -\frac{n}{c_n^2}\log Z_{\alpha, x}^{(n)}
            -\frac{n}{c_n}x\xi
            +\alpha \frac{n^2}{c_n^2} \log \left\{1-x+x e^{c_n\xi/(\alpha n)}\right\}. \nn 
        \end{align}
        Since we have 
            \begin{align*}
                \alpha \frac{n^2}{c_n^2} \log \left\{1-x+x e^{c_n\xi/(\alpha n)}\right\} 
                &=\alpha \frac{n^2}{c_n^2} \log \left\{1-x+x 
                \left( 1+\frac{c_n}{n}\frac{\xi}{\alpha}+
                \frac{c_n^2}{n^2}\frac{\xi^2}{2\alpha^2 }+O\left(\frac{c_n^3}{n^3}\right)\right)\right\} \\
                &=\alpha \frac{n^2}{c_n^2} \log 
                \left(1+\frac{c_n}{n}\frac{x \xi}{\alpha} +\frac{c_n^2}{n^2}\frac{x\xi^2}{2\alpha^2}+O\left(\frac{c_n^3}{n^3}\right)\right) \\
                &= \frac{n}{c_n}x\xi + \frac{1}{2\alpha}x(1-x)\xi^2
                +O\left(\frac{c_n}{n}\right)
            \end{align*}
        as $n \to \infty$, it follows from \eqref{Eq:moderate-rate}
        that 
            \begin{equation}\label{Eq:MDP-for-GE-1}
                \lim_{n \to \infty}
                \frac{n}{c_n^2}\log \widetilde{A}_n^{(1)}
                =\frac{1}{2\alpha}x(1-x)\xi^2. 
            \end{equation}
        Moreover, we have 
        \[
            \frac{n}{c_n^2}\log |\widetilde{A}_n^{(2)}(\omega)| 
            = -\frac{n}{c_n^2}\log Z_{\alpha, x}^{(n)}
            -\frac{n}{c_n}x\xi
            +\alpha \frac{n^2}{c_n^2} \log \left|1-x+x\omega e^{c_n\xi/(\alpha n)}\right|
        \]
        for $\omega \in K_\alpha \setminus \{1\}$. 
        Since 
        \[
        \log \left|1-x+x\omega e^{c_n\xi/(\alpha n)}\right|\to \log|1-x+x\omega|<0
        \quad\text{as $n \to \infty$},
        \]
        we obtain 
        \begin{align}
            &\limsup_{n \to \infty}\frac{n}{c_n^2}\log |\widetilde{A}_n^{(2)}(\omega)|=-\infty < \frac{1}{2\alpha}x(1-x)\xi^2
            \label{Eq:MDP-for-GE-2}
        \end{align}
        for $\omega \in K_\alpha \setminus \{1\}$. 
        On the other hand, by using \eqref{Eq:moderate-rate}, \eqref{Eq:lower-bound-1}, and \eqref{Eq:lower-bound-2}, we have 
            \begin{align}
                \frac{n}{c_n^2}\log |\widetilde{A}_n^{(3)}| 
                &\le \frac{n}{c_n^2}\log 
                \left\{\frac{1}{Z_{\alpha, x}^{(n)}}
                \frac{\alpha e^{-c_n\xi(x+1/n)} x^{-\alpha}(1-x)^{\alpha(n+1)}}{\pi|\sin \alpha \pi|}\int_0^1 t^{\alpha-1}(1-t)^{\alpha n} \, \dd t\right\}  \nn \\
                &\le \frac{n}{c_n^2}\log 
                \left\{\frac{1}{Z_{\alpha, x}^{(n)}}
                \frac{\alpha e^{-c_n\xi(x+1/n)} x^{-\alpha}(1-x)^{\alpha(n+1)}}{\pi|\sin \alpha \pi|}\int_0^1 t^{\alpha-1} \, \dd t\right\} \nn \\
                &=\frac{n}{c_n^2}\log 
                \left\{\frac{1}{Z_{\alpha, x}^{(n)}}
                \frac{e^{-c_n\xi(x+1/n)} x^{-\alpha}(1-x)^{\alpha(n+1)}}{\pi|\sin \alpha \pi|}\right\} \nn \\
                &= -\frac{n}{c_n^2}\log Z_{\alpha, x}^{(n)}
                -\frac{n}{c_n}\left(x+\frac{1}{n}\right)\xi+\alpha\left(\frac{n^2}{c_n^2}
                +\frac{n}{c_n^2}\right)\log(1-x)
                +\frac{n}{c_n^2}\log 
                 \frac{ x^{-\alpha}}{\pi|\sin \alpha \pi|} \nn 
            \end{align} 
        and 
            \begin{align}
                \frac{n}{c_n^2}\log |\widetilde{A}_n^{(4)}| 
                &\le \frac{n}{c_n^2}\log 
                \left\{\frac{1}{Z_{\alpha, x}^{(n)}}
                \frac{\alpha e^{-c_n\xi(x-1-1/n)} x^{\alpha(n+1)}(1-x)^{-\alpha}}{\pi|\sin \alpha \pi|}\int_0^1 t^{\alpha-1}(1-t)^{\alpha n} \, \dd t
                \right\}  \nn \\
                &\le \frac{n}{c_n^2}\log 
                \left\{\frac{1}{Z_{\alpha, x}^{(n)}}
                \frac{\alpha e^{-c_n\xi(x-1-1/n)} x^{\alpha(n+1)} (1-x)^{-\alpha}}{\pi|\sin \alpha \pi|}\int_0^1 t^{\alpha-1} \, \dd t\right\} \nn \\
                &=\frac{n}{c_n^2}\log 
                \left\{\frac{1}{Z_{\alpha, x}^{(n)}}
                \frac{e^{-c_n\xi(x-1-1/n)} x^{\alpha(n+1)} (1-x)^{-\alpha}}{\pi|\sin \alpha \pi|}\right\} \nn \\
                &= -\frac{n}{c_n^2}\log Z_{\alpha, x}^{(n)}
                -\frac{n}{c_n}\left(x-1-\frac{1}{n}\right)\xi
                +\alpha\left(\frac{n^2}{c_n^2}+\frac{n}{c_n^2}\right)\log x
                +\frac{n}{c_n^2}\log 
                \frac{ (1-x)^{-\alpha}}{\pi|\sin \alpha \pi|}.  \nn 
            \end{align} 
        These lead to
            \begin{equation}\label{Eq:MDP-for-GE-3}
                \limsup_{n \to \infty}
                \frac{n}{c_n^2}\log |\widetilde{A}_n^{(j)}|
                =-\infty < \frac{1}{2\alpha}x(1-x)\xi^2, 
                \qquad j=3, 4.   
            \end{equation}
        Therefore, \eqref{Eq:MDP-for-GE-0}, \eqref{Eq:MDP-for-GE-1}, \eqref{Eq:MDP-for-GE-2}, and \eqref{Eq:MDP-for-GE-3} enable us to apply
        Lemma~\ref{Lem:log-inequality} to have the limit
            \[
                \widetilde{\Lambda}_{\alpha, x}(\xi):=
                \lim_{n \to \infty}\frac{n}{c_n^2}\widetilde{\Lambda}_{\alpha, x}^{(n)}\left(\frac{c_n^2}{n}\xi\right)
                =\frac{1}{2\alpha}x(1-x)\xi^2, \qquad \xi \in \R, 
            \]
        with $\mathcal{D}_{\widetilde{\Lambda}_{\alpha, x}}=\R$. 
        This satisfies (a), (b), and (c) in Theorem~\ref{Thm:Gartner-Ellis}-(2), which means that  
        the sequence of laws of $\{Y_{\alpha, x}^{(n)}\}_{n=1}^\infty$
        satisfies the MDP with speed $n/c_n^2$ and  
        quadratic rate function $I_{\alpha, x}^{(\mathrm{M})}$ given by \eqref{Eq:rate-function-LDP}, since 
        the Fenchel--Legendre transform of $\widetilde{\Lambda}_{\alpha, x}$ is 
        \[
        (\widetilde{\Lambda}_{\alpha, x})^*(z)
        =\sup_{\xi \in \R}\left(\xi z - \widetilde{\Lambda}_{\alpha, x}(\xi)\right)=\left.\left(\xi z - \widetilde{\Lambda}_{\alpha, x}(\xi)\right)\right|_{\xi=\frac{\alpha z}{x(1-x)}}
        =\frac{\alpha}{2x(1-x)}z^2. 
        \]
        The case where $\alpha \in \N$ is proved more simply.   
\end{proof}

\section{Berry--Esseen type estimate of fractional binomial distributions}
\label{Sect:Berry-Esseen}

Let $\alpha>0$, $n \in \N$, and $x \in (0, 1)$. 
The purpose of this section is to obtain a
Berry--Esseen type estimate of the fractional binomial distribution
$\mu_{\alpha, x}^{(n)}$ (Theorem~\ref{Thm:Berry-Esseen}). 
Recall that $\widetilde{\mu}_{\alpha, x}^{(n)}$ is the law of 
the normalized random variable $\widetilde{S}_{\alpha, x}^{(n)}$
defined by \eqref{Eq:normalized-mu}. 
The distribution functions of $\widetilde{\mu}_{\alpha, x}^{(n)}$ and 
$N(0, 1)$ are denoted by $\widetilde{F}_{\alpha, x}^{(n)}(z)$ and $\Phi(z)$, respectively. 

At the beginning, we present several elementary but useful inequalities
without proofs.  

\begin{lm}\label{Lem:inequalities}
    \begin{enumerate}
        \item[{\rm (1)}] For $\theta \in \R$, we have 
            \[
                \left|e^{i\theta} - \left( 1+i\theta - \frac{\theta^2}{2}\right)\right|
                \le \frac{1}{6}|\theta|^3.
            \]
        \item[{\rm (2)}] For $z \in \mathbb{C}$, we have 
            $|e^z-1| \le |z|e^{|z|}$.
        \item[{\rm (3)}] For $a \ge 0, b \ge 0$, and $\theta \in \R$, we have 
            $|a-be^{i\theta}| \ge \sqrt{a^2+b^2}|\sin(\theta/2)|$.         
    \end{enumerate}
\end{lm}

\noindent
See \cite[Lemma~2.14]{HN} for the proof of Lemma~\ref{Lem:inequalities}-(3).

The following will play an important role in the proof
of Theorem~\ref{Thm:Berry-Esseen}, which asserts that 
a quantitative estimate of $|\widetilde{F}_{\alpha, x}^{(n)}(z) - \Phi(z)|$ can be reduced to that of their characteristic functions.

\begin{lm}[cf.~{\cite[Lemma~2 in Section~7.4]{Chung}}]
\label{Lem:Chung}
    Let $F \colon \R \to \R$ be a distribution function and 
    $G \colon \R \to \R$ a differentiable function of bounded variation
    satisfying 
        \[
        M:=\dis\sup_{z \in \R}|G'(z)|<\infty \quad \text{and} \quad 
        \int_{\R}|F(z)-G(z)| \, \dd z < \infty.
        \]
    We write 
        \[
            f(\xi)=\int_{\R} e^{i \xi z} \, \dd F(z) \quad \text{and} \quad  
            g(\xi)=\int_{\R} e^{i \xi z} \, \dd G(z)
        \]
    for $\xi \in \R$. 
    Then, we have 
        \[
            \sup_{z \in \R}|F(z)-G(z)| \le \frac{2}{\pi}\int_0^T \frac{|f(\xi)-g(\xi)|}{\xi} \, \dd \xi + \frac{24M}{\pi T}
        \]
    for all $T>0$. 
\end{lm}

Let $\varphi_{\alpha, x}^{(n)}(\xi)=\E[e^{i\xi \widetilde{S}_{\alpha, x}^{(n)}}]$, $\xi \in \R$, be the characteristic function of 
$\widetilde{S}_{\alpha, x}^{(n)}$.  
The following is the key estimate for the proof of Theorem~\ref{Thm:Berry-Esseen}.

\begin{lm}
\label{Lem:CF-estimate} 
   There exist 
   $C_1, C_2, C_3>0$,  
   $\gamma>0$ and $\delta_1>0$
   such that
    \begin{align*}
        \left| \widetilde{\varphi}_{\alpha, x}^{(n)}(\xi) - e^{-\xi^2/2}\right|
        &\le 
        C_{1}\left( \xi e^{-\delta_1 n}+\xi^3n^{-1/2}\right)
        \exp\left(C_2\xi-\frac{1}{6}\xi^2\right)
        +C_3\xi e^{-\delta n}
    \end{align*}
   for $n \in \N$ and $0 \le \xi \le \gamma\sqrt{n}$.
\end{lm}

\begin{proof}
    We split the proof into several steps. 
\smallskip

    \noindent
    {\bf Step~1.} We put 
    \[
    \theta_\alpha:=\begin{cases}
        \{(\alpha-\underline{\alpha}) \wedge (\overline{\alpha}-\alpha)\}\pi
        & \text{if }\alpha \notin \N \\
        2\pi & \text{if }\alpha \in \N,
    \end{cases}
    \]
    where $\underline{\alpha}=2\lfloor \alpha/2 \rfloor$ and 
    $\overline{\alpha}=2\lceil \alpha/2 \rceil$ for $\alpha>0$.
    It follows from \cite[Theorem~2.15]{HN} that
        \begin{align}
            &\varphi_{\alpha, x}^{(n)}(\xi)\nn\\*
            &= \frac{1}{Z_{\alpha, x}^{(n)}}\left(1-x+xe^{i\xi/\alpha}\right)^{\alpha n}
            +\frac{1}{Z_{\alpha, x}^{(n)}}
            \sum_{\omega \in K_\alpha \setminus \{1\}}
            \left(1-x+x\omega e^{i\xi/\alpha}\right)^{\alpha n} \nn \\
            &\quad-\frac{1}{Z_{\alpha, x}^{(n)}}
            \frac{\alpha \sin \alpha \pi}{\pi}
            \int_0^1 t^{\alpha-1}(1-t)^{\alpha n}
            \left\{
            \frac{x^\alpha (1-x)^{\alpha(n+1)}e^{i\xi}}{\psi_{\alpha, x}^{(1)}(t, \xi)\psi_{\alpha, x}^{(2)}(t, \xi)}
            +\frac{x^{\alpha(n+1)} (1-x)^{\alpha}e^{i\xi(n-1)}}{\psi_{\alpha, x}^{(3)}(t, \xi)\psi_{\alpha, x}^{(4)}(t, \xi)}
            \right\} \, \dd t 
            \label{Eq:fractional-CF-explicit}
         \end{align}
    for $\xi \in (-\theta_\alpha, \theta_\alpha)$, where we put
        \begin{align*}
            \psi_{\alpha, x}^{(1)}(t, \xi) &= \{t(1-x)\}^\alpha - x^\alpha e^{i(\xi-\alpha \pi)}, \\
            \psi_{\alpha, x}^{(2)}(t, \xi) &= \{t(1-x)\}^\alpha - x^\alpha e^{i(\xi+\alpha \pi)}, \\
            \psi_{\alpha, x}^{(3)}(t, \xi) &= (1-x)^\alpha e^{-i(\xi+\alpha \pi)}-(tx)^\alpha, \\
            \psi_{\alpha, x}^{(4)}(t, \xi) &= (1-x)^\alpha e^{-i(\xi-\alpha \pi)}-(tx)^\alpha,
        \end{align*}
    for $t \in (0, 1)$ and $\xi \in \R$.
    In what follows, we set 
    $m_n=\mathbb{E}[S_{\alpha, x}^{(n)}]$ and 
    $v_n=\mathrm{Var}(S_{\alpha, x}^{(n)})$ for simplicity. 
    Since it holds that 
        \[
        \widetilde{\varphi}_{\alpha, x}^{(n)}(\xi)
        =\exp\left(-\frac{i\xi m_n}{\sqrt{v_n}}\right)\varphi_{\alpha, x}^{(n)}\left(\frac{\xi}{\sqrt{v_n}}\right),
        \qquad \xi \in \R,
        \]
    we have 
        \[
            \widetilde{\varphi}_{\alpha, x}^{(n)}(\xi)
            =\frac{1}{Z_{\alpha, x}^{(n)}}
            \left(I_{\alpha, x}^{(n)}(\xi) + J_{\alpha, x}^{(n)}(\xi)\right),
            \qquad |\xi|<\theta_\alpha \sqrt{v_n},
        \]
    where
        \begin{align}
            I_{\alpha, x}^{(n)}(\xi)
            &= \exp\left(-\frac{i\xi m_n}{\sqrt{v_n}}\right)
            \left\{1-x+x\exp\left(\frac{i\xi}{\alpha\sqrt{v_n}}\right)\right\}^{\alpha n},\nn\\
            J_{\alpha, x}^{(n)}(\xi)
            &= \exp\left(-\frac{i\xi m_n}{\sqrt{v_n}}\right)\Biggl[
            \sum_{\omega \in K_\alpha \setminus \{1\}}
            \left\{1-x+x\omega\exp\left(\frac{i\xi}{\alpha\sqrt{v_n}}\right)\right\}^{\alpha n} \nn\\
            &\quad-\frac{\alpha \sin \alpha \pi}{\pi}\int_0^1 
            t^{\alpha-1}(1-t)^{\alpha n} \nn\\
            &\quad\times
            \left\{
            \frac{x^\alpha (1-x)^{\alpha(n+1)}e^{i\xi/\sqrt{v_n}}}{\psi_{\alpha, x}^{(1)}(t, \xi/\sqrt{v_n})\psi_{\alpha, x}^{(2)}(t, \xi/\sqrt{v_n})}
            +\frac{x^{\alpha(n+1)} (1-x)^{\alpha}e^{i\xi(n-1)/\sqrt{v_n}}}{\psi_{\alpha, x}^{(3)}(t, \xi/\sqrt{v_n})\psi_{\alpha, x}^{(4)}(t, \xi/\sqrt{v_n})}
            \right\} \, \dd t\Biggr].
            \label{Eq:J}
        \end{align}
    Then, the equality 
        \[
        Z_{\alpha, x}^{(n)}=I_{\alpha, x}^{(n)}(0)+J_{\alpha, x}^{(n)}(0)=1+J_{\alpha, x}^{(n)}(0)
        \]
    allows us to obtain
        \begin{align}
            \left|\widetilde{\varphi}_{\alpha, x}^{(n)}(\xi) - e^{-\xi^2/2}\right| 
            &=\left|\frac{1}{Z_{\alpha, x}^{(n)}}\left\{ 
            I_{\alpha, x}^{(n)}(\xi) + J_{\alpha, x}^{(n)}(\xi)
            - \left(1+J_{\alpha, x}^{(n)}(0) \right)e^{-\xi^2/2}\right\}\right| \nn \\
            &\le C_4 \left\{\left|
            I_{\alpha, x}^{(n)}(\xi) - e^{-\xi^2/2}\right|+
            \left|J_{\alpha, x}^{(n)}(0)(1-e^{-\xi^2/2})\right|
            +\left|J_{\alpha, x}^{(n)}(\xi) - J_{\alpha, x}^{(n)}(0)\right|\right\}
            \label{Eq:CF-difference}
        \end{align}
    for $|\xi|<\theta_\alpha \sqrt{v_n}$.
\smallskip

    \noindent
    {\bf Step~2}.
    Let us estimate the term 
    $|I_{\alpha, x}^{(n)}(\xi) - e^{-\xi^2/2}|$ 
    on the right-hand side of \eqref{Eq:CF-difference}. 
    We have 
        \begin{align*}
              I_{\alpha, x}^{(n)}(\xi)
              &= \left[(1-x)\exp\left(-\frac{i\xi m_n}{\alpha n \sqrt{v_n}}\right)+x\exp\left\{\frac{i\xi}{\alpha\sqrt{v_n}}\left(1-\frac{m_n}{n}\right)\right\}\right]^{\alpha n} \nn \\
              &=\left[ (1-x) \left\{ 1-\frac{i\xi m_n}{\alpha n\sqrt{v_n}}
              -\frac{\xi^2 m_n^2}{2(\alpha n\sqrt{v_n})^2}
              +O\left( \left(\frac{\xi m_n}{\alpha n\sqrt{v_n}}\right)^3\right)\right\}\right. \nn \\
              &\qquad+x\left.\left\{1+\frac{i\xi}{\alpha\sqrt{v_n}}\left(1-\frac{m_n}{n}\right)
              -\frac{\xi^2}{2(\alpha \sqrt{v_n})^2}\left(1-\frac{m_n}{n}\right)^2+O\left( \left( \frac{\xi}{\alpha\sqrt{v_n}}\right)^3
              \left( 1-\frac{m_n}{n}\right)^3\right)\right\} \right]^{\alpha n}
        \end{align*}
    as $n \to \infty$. In view of \eqref{Eq:mean-variance-fractional}, 
    it holds that 
        \begin{align*}
            -(1-x)\frac{i\xi m_n}{\alpha n \sqrt{v_n}}
            +x\frac{i\xi}{\alpha \sqrt{v_n}}\left(1-\frac{m_n}{n}\right)
            &=\frac{i\xi}{\alpha \sqrt{v_n}}\left(x-\frac{m_n}{n}\right)\\
            &=\frac{\xi}{\sqrt{v_n}}O(e^{-\delta n})=\xi O(n^{-1/2}e^{-\delta n})
        \end{align*}
    and
        \begin{align*}
            -(1-x)\frac{\xi^2 m_n^2}{2(\alpha n\sqrt{v_n})^2} 
            - x\frac{\xi^2}{2(\alpha \sqrt{v_n})^2}\left(1-\frac{m_n}{n}\right)^2 
            &=-\frac{\xi^2}{2(\alpha\sqrt{v_n})^2}\left\{
            x(1-x)+\left(\frac{m_n}{n}-x\right)^2\right\}\\
            &=-\frac{\xi^2}{2\alpha n}+\xi^2O(n^{-1}e^{-\delta n})
        \end{align*}
    as $n \to \infty$, which leads to 
        \[
            I_{\alpha, x}^{(n)}(\xi)
            = \left\{ 1 - \frac{\xi^2}{2\alpha n}
            +(|\xi|+\xi^2)O(e^{-\delta n}) + |\xi|^3O(n^{-3/2})\right\}^{\alpha n}
        \]
    as $n \to \infty$. 
    Due to $\mathrm{Log}\,(1+t)=t+O(t^2)$ as $t \to 0$, we have 
        \begin{align*}
            \mathrm{Log}\,I_{\alpha, x}^{(n)}(\xi)
            &=\alpha n \Biggl[ - \frac{\xi^2}{2\alpha n}
            +(\xi+\xi^2)O(e^{-\delta n}) + \xi^3O(n^{-3/2}) \nn \\
            &\qquad
            +O\left( \left\{ - \frac{\xi^2}{2\alpha n}
            +(\xi+\xi^2)O(e^{-\delta n}) + \xi^3 O(n^{-3/2}) \right\}^2\right)
            \Biggr] \nn \\
            &=-\frac{\xi^2}{2}+(\xi+\xi^2)O(e^{-\delta_2 n})+\xi^3O(n^{-1/2})+\xi^4O(n^{-1})+(\xi+\xi^2)^2O(e^{-\delta_3 n})+\xi^6O(n^{-2}) \nn \\
            &=-\frac{\xi^2}{2}+\xi O(e^{-\delta_4 n})+\xi^3O(n^{-1/2})
        \end{align*}
    as $n \to \infty$ for $0 \le \xi< \theta_\alpha \sqrt{v_n}$, where we used 
    $\xi=O(n^{1/2})$ as $n \to \infty$ in the final line. 
    Thus, we can take sufficiently small $\gamma \in (0, \theta_\alpha)$ satisfying  
        \[
            \left|\mathrm{Log}\, I_{\alpha, x}^{(n)}(\xi)+\frac{\xi^2}{2}\right|
            \le C_5\xi+\frac{1}{3}\xi^2, \qquad 
            0 \le \xi \le \gamma \sqrt{n}.
        \]
    Then, Lemma~\ref{Lem:inequalities}-(2) gives 
        \[
            \left| I_{\alpha, x}^{(n)}(\xi)e^{\xi^2/2}-1\right|
            \le C_{6}\left( \xi e^{-\delta_4 n}+\xi^3n^{-1/2}\right)
            \exp\left(C_5\xi+\frac{1}{3}\xi^2\right)
        \]
    and so 
        \begin{equation}
            \left| I_{\alpha, x}^{(n)}(\xi)-e^{-\xi^2/2}\right|
            \le C_{6}\left( \xi e^{-\delta_4 n}+\xi^3n^{-1/2}\right)
            \exp\left(C_5\xi-\frac{1}{6}\xi^2\right)
            \label{Eq:estimatee-(I-e)}
        \end{equation}
    for $0 \le \xi \le \gamma \sqrt{n}$. 
\smallskip

    \noindent
    {\bf Step~3}. 
    We next deal with the estimate of $J_{\alpha, x}^{(n)}(\xi)$.
    First, suppose that $\alpha\not\in\N$.
    Since $|\xi|<\theta_\alpha$ implies $\sin(\xi \pm \alpha \pi)/2 \neq 0$,
    we apply Lemma~\ref{Lem:inequalities}-(3) to deduce
        \[
            |\{t(1-x)\}^\alpha-x^\alpha e^{i(\xi \pm \alpha \pi)}| \ge x^\alpha \left|\sin \frac{\xi \pm \alpha \pi}{2}\right|>0
        \]
    and     
        \[
            |(1-x)^\alpha e^{-i(\xi \pm \alpha \pi)}-(tx)^\alpha| \ge (1-x)^\alpha \left|\sin \frac{\xi \pm \alpha \pi}{2}\right|>0.
        \]
    Therefore, there exists some $C_{7}>0$ such that  
    $|\psi_{\alpha, x}^{(j)}(t, \xi)| \ge C_{7}$ for $t \in (0, 1)$, $|\xi| \le \theta_\alpha/2$, and $j=1, 2, 3, 4$. 
    In particular, $|\xi| \le \theta_\alpha\sqrt{v_n}/2$ implies 
        \begin{equation}
            \left|\psi_{\alpha, x}^{(n)}\left(t,\frac{\xi}{\sqrt{v_n}}\right)\right| \ge C_{7}, \qquad 
            t \in (0, 1), \ j=1, 2, 3, 4.
            \label{Eq:psi-estimate}
        \end{equation}
    If necessary, we may take $\gamma>0$ in Step~2 as satisfying 
        \[
            \gamma\sqrt{n} \le \frac{1}{2}\theta_\alpha\sqrt{v_n}, \qquad n \in \N. 
        \]
    By noting that 
        \begin{align}
            \left|1-x+x\omega \exp\left(\frac{i\xi}{\alpha \sqrt{v_n}}\right)\right|^2
            &=(1-x)^2+2x(1-x)\mathrm{Re}\left(\omega\exp\left(\frac{i\xi}{\alpha \sqrt{v_n}}\right)\right)+x^2 \nn \\
            &=1-2x(1-x)\left\{ 1-\mathrm{Re}\left(\omega\exp\left(\frac{i\xi}{\alpha \sqrt{v_n}}\right)\right)\right\} \nn \\
            &\le 1-2x(1-x)\left\{ 1-\mathrm{Re}\left(\omega\exp\left(\frac{i \theta_\alpha}{2\alpha}\right)\right)\right\}<1
            \label{Eq:estimate<1}
        \end{align}
    for $\omega \in K_\alpha \setminus \{1\}$ and $|\xi| \le \gamma \sqrt{n}$,
    we obtain
        \begin{align}
            |J_{\alpha, x}^{(n)}(\xi)|
            &\le \sum_{\omega \in K_\alpha \setminus \{1\}}
            \left| 1-x+x\omega \exp\left(\frac{i\xi}{\alpha \sqrt{v_n}}\right)\right|^{\alpha n}
            +\frac{\alpha |\sin \alpha \pi|}{\pi}\int_0^1 t^{\alpha-1}(1-t)^{\alpha n} \nn \\
            &\qquad
            \times\left\{
            \frac{x^\alpha (1-x)^{\alpha(n+1)}}{|\psi_{\alpha, x}^{(1)}(t, \xi/\sqrt{v_n})||\psi_{\alpha, x}^{(2)}(t, \xi/\sqrt{v_n})|}
            +\frac{x^{\alpha(n+1)} (1-x)^{\alpha}}{|\psi_{\alpha, x}^{(3)}(t, \xi/\sqrt{v_n})||\psi_{\alpha, x}^{(4)}(t, \xi/\sqrt{v_n})|}
            \right\} \, \dd t \nn \\
            &\le C_{8}e^{-\delta_5 n}
            \label{Eq:estimate-J}
        \end{align}
    for $n \in \N$ and $|\xi| \le \gamma \sqrt{n}$.
    The estimate of the same type holds also when $\alpha\in\N$ since \eqref{Eq:estimate<1} is valid in this case and the second term of \eqref{Eq:J} vanishes.
\smallskip

    \noindent
    {\bf Step~4}. 
    We give upper bounds of the terms $|J_{\alpha, x}^{(n)}(0)(1-e^{-\xi^2/2})|$ and $|J_{\alpha, x}^{(n)}(\xi)-J_{\alpha, x}^{(n)}(0)|$ 
    on the right-hand side of \eqref{Eq:CF-difference}. 
       By combining \eqref{Eq:estimate-J} with Lemma~\ref{Lem:inequalities}-(2), 
        \begin{equation}
            \left| J_{\alpha, x}^{(n)}(0)(1-e^{-\xi^2/2})\right|
            \le C_{9}\xi^2 e^{-\xi^2/2}e^{-\delta_5 n}  \le C_9|\xi|e^{-\delta_5 n}
            \label{Eq:estimate-J(1-e)}
        \end{equation}
    for $n \in \N$ and $|\xi| \le \gamma \sqrt{n}$.
    We also have 
        \begin{align}
            |J_{\alpha, x}^{(n)}(\xi)-J_{\alpha, x}^{(n)}(0)|
            &\le |\xi| \sup_{s \in [0 \wedge \xi, 0 \vee \xi]} \left| \frac{\dd J_{\alpha, x}^{(n)}}{\dd \xi}(s)\right|. 
            \label{Eq:estimate-JJ-1}
        \end{align}
    For $|\xi| \le \gamma \sqrt{n}$, a direct computation gives    
        \begin{align}
            \frac{\dd J_{\alpha, x}^{(n)}}{\dd \xi}(\xi)
            &=-\frac{im_n}{\sqrt{v_n}}J_{\alpha, x}^{(n)}(\xi) 
            +\exp\left(-\frac{i\xi m_n}{\sqrt{v_n}}\right) \nn \\*
            &\quad\times\Biggl[ \sum_{\omega \in K_\alpha \setminus \{1\}}
            \frac{inx\omega}{\sqrt{v_n}}\exp\left(\frac{i\xi}{\alpha\sqrt{v_n}}\right)
            \left\{1-x+x\omega \exp\left(\frac{i\xi}{\alpha \sqrt{v_n}}\right)\right\}^{\alpha n-1} \nn \\*
            &\quad
            -\frac{\alpha \sin \alpha \pi}{\pi}\int_0^1 t^{\alpha-1}(1-t)^{\alpha n}\Biggl\{\frac{x^\alpha(1-x)^{\alpha(n+1)}e^{i\xi/\sqrt{v_n}}}{\sqrt{v_n}}A_{\alpha, x}^{(n)}(t, \xi) \nn \\*
            &\quad
            +\frac{x^{\alpha(n+1)}(1-x)^{\alpha}e^{i\xi(n-1)/\sqrt{v_n}}}{\sqrt{v_n}}B_{\alpha, x}^{(n)}(t, \xi)\Biggr\}\,\dd t\Biggr],
            \label{Eq:J-derivative}
        \end{align}
    where we put 
        \begin{align*}
            A_{\alpha, x}^{(n)}(t, \xi)
            &=\frac{i}{\psi_{\alpha, x}^{(1)}(t, \xi/\sqrt{v_n})\psi_{\alpha, x}^{(2)}(t, \xi/\sqrt{v_n})}
            -\frac{(\del\psi_{\alpha, x}^{(1)}/\del \xi)(t, \xi/\sqrt{v_n})}{\psi_{\alpha, x}^{(1)}(t, \xi/\sqrt{v_n})^2\psi_{\alpha, x}^{(2)}(t, \xi/\sqrt{v_n})} \nn \\
            &\quad
            -\frac{(\del\psi_{\alpha, x}^{(2)}/\del \xi)(t, \xi/\sqrt{v_n})}{\psi_{\alpha, x}^{(1)}(t, \xi/\sqrt{v_n})\psi_{\alpha, x}^{(2)}(t, \xi/\sqrt{v_n})^2}, \nn \\
            B_{\alpha, x}^{(n)}(t, \xi)
            &=\frac{i(n-1)}{\psi_{\alpha, x}^{(3)}(t, \xi/\sqrt{v_n})\psi_{\alpha, x}^{(4)}(t, \xi/\sqrt{v_n})}
            -\frac{(\del\psi_{\alpha, x}^{(3)}/\del \xi)(t, \xi/\sqrt{v_n})}{\psi_{\alpha, x}^{(3)}(t, \xi/\sqrt{v_n})^2\psi_{\alpha, x}^{(4)}(t, \xi/\sqrt{v_n})} \nn \\
            &\quad
            -\frac{(\del\psi_{\alpha, x}^{(4)}/\del \xi)(t, \xi/\sqrt{v_n})}{\psi_{\alpha, x}^{(3)}(t, \xi/\sqrt{v_n})\psi_{\alpha, x}^{(4)}(t, \xi/\sqrt{v_n})^2}.
        \end{align*}
    By noting 
        \begin{align*}
            \frac{\del\psi_{\alpha, x}^{(1)}}{\del \xi}(t, \xi) 
            &= -ix^\alpha e^{i(\xi-\alpha \pi)}, & 
            \frac{\del\psi_{\alpha, x}^{(2)}}{\del \xi}(t, \xi) 
            &= -ix^\alpha e^{i(\xi+\alpha \pi)}, \\
            \frac{\del\psi_{\alpha, x}^{(3)}}{\del \xi}(t, \xi) 
            &= -i(1-x)^\alpha e^{-i(\xi+\alpha \pi)}, &
            \frac{\del\psi_{\alpha, x}^{(4)}}{\del \xi}(t, \xi) 
            &= -i(1-x)^\alpha e^{-i(\xi-\alpha \pi)},
        \end{align*}
    we see that 
        \begin{equation}
            \left|\frac{\del \psi_{\alpha, x}^{(j)}}{\del \xi}(\xi)\right| \le 1,
             \qquad j=1, 2, 3, 4.
             \label{Eq:psi-derivative-estimate}
        \end{equation}
    Hence, for $|\xi| \le \gamma \sqrt{n} $, it follows from 
    \eqref{Eq:psi-estimate} and \eqref{Eq:psi-derivative-estimate} that
        \begin{align*}
            |A_{\alpha, x}^{(n)}(t, \xi)| &\le \frac{1}{C_{7}^2}
            +\frac{1}{C_{7}^3}+\frac{1}{C_7^3} \le C_{10}, \\
            |B_{\alpha, x}^{(n)}(t, \xi)| &\le \frac{n-1}{C_7^2}
            +\frac{1}{C_7^3}+\frac{1}{C_7^3} \le C_{11}n.
        \end{align*}
    Precisely speaking, the exchange of derivation and integration 
    in \eqref{Eq:J-derivative} is justified by these estimates. 
    These estimates together with 
    \eqref{Eq:estimate<1} and \eqref{Eq:estimate-J} allow us to have
        \begin{align*}
            \left|\frac{\dd J_{\alpha, x}^{(n)}}{\dd \xi}(\xi)\right|
            &\le C_{12}\sqrt{n}e^{-\delta_5 n}
            +\sum_{\omega \in K_\alpha \setminus \{1\}} C_{13}\sqrt{n}e^{-\delta_5 n}
            +\frac{C_{14}}{\sqrt{n}}\int_0^1 t^{\alpha-1}(1-t)^{\alpha n}\nn \\
            &\qquad
            \times \left\{x^\alpha(1-x)^{\alpha(n+1)} C_{10} + x^{\alpha(n+1)}(1-x)^\alpha C_{11}n\right\} \, \dd t \nn \\
            &\le C_{15}\sqrt{n}e^{-\delta_5 n}
            +C_{16}\sqrt{n}e^{-\delta_6 n}\int_0^1 t^{\alpha-1} \, \dd t \nn \\
            &\le C_{17}e^{-\delta_7 n}
        \end{align*}
    for $|\xi| \le \gamma \sqrt{n}$. 
    By combining this estimate with \eqref{Eq:estimate-JJ-1}, we obtain 
        \begin{equation}
            |J_{\alpha, x}^{(n)}(\xi)-J_{\alpha, x}^{(n)}(0)|
            \le  C_{17}\xi e^{-\delta_7 n}, \qquad  0 \le \xi \le \gamma \sqrt{n}. 
            \label{Eq:estimate-JJ-2}
        \end{equation}

Consequently, it follows from \eqref{Eq:CF-difference}, 
\eqref{Eq:estimatee-(I-e)},
\eqref{Eq:estimate-J(1-e)}, and \eqref{Eq:estimate-JJ-2} that 
    \[
        \left| \widetilde{\varphi}_{\alpha, x}^{(n)}(\xi) - e^{-\xi^2/2}\right|
        \le C_{4}\left\{
        C_{6}\left( \xi e^{-\delta_4 n}+\xi^3n^{-1/2}\right)
        \exp\left(C_5\xi-\frac{1}{6}\xi^2\right)
        +(C_9+C_{17})\xi e^{-\delta_8 n}
        \right\}
    \]
for $0 \le \xi \le \gamma\sqrt{n}$, 
which is the very desired estimate. 
\end{proof}

We are now in a position to prove Theorem~\ref{Thm:Berry-Esseen}.

\begin{proof}[Proof of Theorem~\ref{Thm:Berry-Esseen}]
By making use of the Chebyshev inequality, we have 
    \begin{align*}
        \widetilde{F}_{\alpha, x}^{(n)}(z) \vee \Phi(z) &\le \frac{1}{z^2} \qquad \text{if }z\le 0, \\
        \left(1-\widetilde{F}_{\alpha, x}^{(n)}(z)\right)
        \vee (1-\Phi(z)) &\le \frac{1}{z^2} \qquad \text{if }z> 0.
    \end{align*}
Thus, we have 
    \[
        \left|\widetilde{F}_{\alpha, x}^{(n)}(z)-\Phi(z)\right| \le \frac{1}{z^2}, 
        \qquad  z \in \R. 
    \]
Since $|\widetilde{F}_{\alpha, x}^{(n)}(z)| \le 1$ 
and $|\Phi(z)|\le 1$ for all $z \in \R$, it holds that 
    \[
        \int_{\R} \left|\widetilde{F}_{\alpha, x}^{(n)}(z) - \Phi(z)\right| \, \dd z 
        \le \int_{|z|\le 1} 2 \, \dd z + \int_{|z|>1}\frac{\dd z}{z^2}=6<\infty.
    \]
Moreover, it is obvious that 
    \[
        \Phi'(z)=\frac{1}{\sqrt{2\pi}}e^{-z^2/2} \le \frac{1}{\sqrt{2\pi}}, \qquad  z \in \R. 
    \]
Hence, Lemma~\ref{Lem:CF-estimate} and Lemma~\ref{Lem:Chung} with $F=\widetilde{F}_{\alpha, x}^{(n)}$, $G=\Phi$, $M=1/\sqrt{2\pi}$,
$f(\xi)=\widetilde{\varphi}_{\alpha, x}^{(n)}(\xi)$, $g(\xi)=e^{-\xi^2/2}$, 
and $T=\gamma\sqrt{n}$
allow us to obtain 
    \begin{align*}
        \sup_{z \in \R}\left|\widetilde{F}_{\alpha, x}^{(n)}(z) - \Phi(z)\right|
        &\le \frac{2}{\pi}\int_0^{\gamma \sqrt{n}}
        \frac{\left|\widetilde{\varphi}_{\alpha, x}^{(n)}(\xi) - e^{-\xi^2/2}\right|}{\xi} \, \dd \xi + \frac{24}{\gamma\pi\sqrt{2\pi}}\frac{1}{\sqrt{n}} \nn \\
        &\le C_{18}\int_0^{\gamma\sqrt{n}}
        \left\{\left( e^{-\delta_1 n}+\xi^2n^{-1/2}\right)
        \exp\left(C_2\xi-\frac{1}{6}\xi^2\right) 
        +e^{-\delta_1 n}\right\} \, \dd \xi + \frac{C_{19}}{\sqrt{n}}
        \nn \\
        &\le C_{20}\left( e^{-\delta_1 n}+\frac{1}{\sqrt{n}}+\gamma\sqrt{n}e^{-\delta_1 n}\right)+\frac{C_{19}}{\sqrt{n}} \nn \\
        &\le \frac{C_{21}}{\sqrt{n}}
    \end{align*}
    for $n \in \N$. 
\end{proof}

\section{Concluding remarks}
\label{Sect:Remarks}

We give a few comments and supplementary propositions as concluding remarks.

\begin{enumerate}
\item[{\rm (1)}] It would be an interesting problem 
to investigate how small we can take the positive constant $C$ 
in the Berry--Esseen type estimate \eqref{Eq:BE-estimate}. 
The proof of Theorem~\ref{Thm:Berry-Esseen} shows that $C$ can be taken uniformly for $x\in[\ve,1-\ve]$ and $\alpha\in[\ve,1/\ve]\setminus\bigcup_{n\in\N}\bigl((2n-\ve,2n)\cup(2n,2n+\ve)\bigr)$ for each $\ve\in(0,1)$, but the constant given in the proof is far from the optimal one.
Let us review the classical situation for reference.
Let $\{X_{n}\}_{n=1}^{\infty}$ be a sequence of 
independent real random variables with finite third moment 
and put $\rho=\sigma^{2}/\beta$, where
    \[
        \sigma^{2}=\frac{1}{n}\sum_{k=1}^{n}\Var(X_{k})
        \quad\text{and}\quad
        \beta=\frac{1}{n}\sum_{k=1}^{n}\E\left[\left|X_{k}-\E[X_{k}]\right|^{3}\right].
    \]
Let $F_{n}(z)$ be the distribution function of $(\sqrt{n\sigma^{2}})^{-1}\sum_{k=1}^{n}(X_{k}-\E[X_{k}])$ for $n\in\N$.
Esseen obtained in \cite{Esseen} not only the bound
\[
\sup_{z\in\R}|F_{n}(z)-\Phi(z)|\le\frac{C\rho}{\sqrt{n}}, \qquad n \in \N, 
\]
but also its asymptotically optimal constant
\[
\lim_{n\rightarrow\infty}\sup_{F_{n}}\left\{\frac{\sqrt{n}}{\rho}\sup_{z\in\R}|F_{n}(z)-\Phi(z)|\right\}=\frac{3+\sqrt{10}}{6\sqrt{2}\pi}, 
\]
where $\sup_{F_{n}}$ denotes the supremum taken over all independent random variables with finite third moment. 
Since it is not expected that $\mu_{\alpha,x}^{(n)}$ can be expressed 
as the law of the sum of independent random variables, 
we need other techniques to obtain the best constant for $\mu_{\alpha,x}^{(n)}$.

\item[{\rm (2)}] 
Throughout the present paper and the previous paper~\cite{HN}, 
we have shown various asymptotic behaviors of $\mu_{\alpha,x}^{(n)}$ 
such as
\begin{enumerate}
\item[(i)]
the weak law of large numbers~(\cite[Theorem~3.1]{HN}), 
\item[(ii)] 
the central limit theorem~(\cite[Theorem~3.2]{HN}), 
\item[(iii)] 
the law of small numbers~(\cite[Theorem~3.4]{HN}), 
\item[(iv)] 
the large deviation principle~(Theorem~\ref{Thm:LDP-fractional}), 
\item[(v)] 
the moderate deviation principle~(Theorem~\ref{Thm:MDP-fractional}), and
\item[(vi)]
a Berry--Esseen type estimate (Theorem~\ref{Thm:Berry-Esseen}).
\end{enumerate}
We note that there exists a probability distribution 
derived from the binomial distribution that have similar behaviors to $\mu_{\alpha,x}^{(n)}$.
Let $X_{\alpha, x}^{(n)}$ be a random variable obeying 
the binomial distribution~$\mathrm{Bin}(\lfloor\alpha n\rfloor,x)$ 
and let $\nu_{\alpha,x}^{(n)}$ be the law of 
$X_{\alpha, x}^{(n)}/\alpha$. 
In other words, $\nu_{\alpha,x}^{(n)}$ is the 
$\lfloor\alpha n\rfloor$-times convolutions of the probability measure 
$(1-x)\delta_0+x\delta_{1/\alpha}$. 
Then, the mean and variance of $\nu_{\alpha,x}^{(n)}$ are $\lfloor\alpha n\rfloor x/\alpha$ and $\lfloor\alpha n\rfloor x(1-x)/\alpha^2$ respectively, which are asymptotically equal to those of $\mu_{\alpha,x}^{(n)}$ as $n\to\infty$ in view of \eqref{Eq:mean-variance-fractional}.
From the classical results, $\nu_{\alpha,x}^{(n)}$ has 
the limit behaviors of the same type as (i)--(vi) 
above \emph{except for} (iii). 
We also note the following similarity. 
From \eqref{Eq:fractional-CF-explicit}, 
the characteristic function $\varphi_{\alpha,x}^{(n)}(\xi)$ of $\mu_{\alpha,x}^{(n)}$ is expressed as
    \[
        \varphi_{\alpha,x}^{(n)}(\xi)=
        \left(1-x+xe^{i\xi/\alpha}\right)^{\alpha n}\left(1+O(e^{-\delta n})\right)+O(e^{-\delta n}),
        \qquad \xi\in(-\theta_{\alpha},\theta_{\alpha}),
    \]
as $n\rightarrow\infty$ for some $\delta>0$. 
On the other hand, the characteristic function $\psi_{\alpha,x}^{(n)}(\xi)$ of $\nu_{\alpha,x}^{(n)}$ 
is given by
\[
\psi_{\alpha,x}^{(n)}(\xi)={\left(1-x+xe^{i\xi/\alpha}\right)}^{\lfloor\alpha n\rfloor}, \qquad \xi\in\R.
\]
Therefore, $\varphi_{\alpha,x}^{(n)}$ is close to $\psi_{\alpha,x}^{(n)}$ 
on the interval $(-\theta_{\alpha},\theta_{\alpha})$ 
when $n$ is sufficiently large. 
These observations may allow us to say that $\mu_{\alpha,x}^{(n)}$ is close to $\nu_{\alpha,x}^{(n)}$ when $n$ is sufficiently large, 
although the support $\{0,1,2,\dots,n\}$ of $\mu_{\alpha,x}^{(n)}$ 
is different from the support 
$\{0,1/\alpha,2/\alpha,\dots,\lfloor\alpha n\rfloor/\alpha\}$ 
of $\nu_{\alpha,x}^{(n)}$. 
\end{enumerate}

As final remarks, we note the following estimates on the moments 
and the distribution functions of $\mu_{\alpha,x}^{(n)}$ 
and $\nu_{\alpha,x}^{(n)}$.

    \begin{pr}\label{Prop:difference_of_moments}
        Let $m \in \N$. Then, we have 
            \begin{equation}\label{Eq:difference_of_moments}
                \int_{\R}z^m\,\mu_{\alpha,x}^{(n)}(\dd z)-\int_{\R}z^m\,\nu_{\alpha,x}^{(n)}(\dd z)
                =O(n^{m-1})
            \end{equation}
        as $n \to \infty$. 
    \end{pr}
    
\begin{proof}
    From \cite[Proposition~2.6]{HN}, 
    there exists $\delta>0$ such that
        \begin{equation}\label{Eq:1}
            \alpha \sum_{j=0}^n (\alpha j)^m 
            \binom{\alpha n}{\alpha j}x^{\alpha j}(1-x)^{\alpha(n-j)}
            =\sum_{k=0}^m W_0(m, k)(\alpha n)_kx^k + O(e^{-\delta n})
        \end{equation}
    as $n\to\infty$. 
    Here, $W_0(m, k)$ is some constant defined in \cite[Lemma~2.3]{HN}, and for $s>0$,
    \[
    (s)_0=1,\quad (s)_k=s(s-1)\cdots(s-k+1),\quad k\in\N.
    \]
    On the other hand, from a classical result (or by applying \cite[Theorem~2.5]{HN} with $n$ and $\alpha$ replaced by $\lfloor\alpha n\rfloor$ and $1$, respectively), we have 
        \begin{equation}\label{Eq:2}
            \sum_{j=0}^{\lfloor\alpha n\rfloor} j^m 
            \binom{\lfloor\alpha n\rfloor}{j}x^{j}(1-x)^{\lfloor\alpha n\rfloor-j}
            =\sum_{k=0}^m W_0(m, k)(\lfloor\alpha n\rfloor)_kx^k.
        \end{equation}
    By combining \eqref{Eq:1} and \eqref{Eq:2}, it follows that
        \[
            Z_{\alpha,x}^{(n)}\alpha^m\int_{\R}z^m\,\mu_{\alpha,x}^{(n)}(\dd z)-\alpha^m\int_{\R}z^m\,\nu_{\alpha,x}^{(n)}(\dd z)
            =\sum_{k=1}^m W_0(m, k)\{(\alpha n)_k-(\lfloor\alpha n\rfloor)_k\}x^k+ O(e^{-\delta n})
        \]
    as $n \to \infty$.
    Since $Z_{\alpha,x}^{(n)}=1+O(e^{-\delta' n})$  
    for some $\delta'>0$ 
    and $(\alpha n)_k-(\lfloor\alpha n\rfloor)_k=O(n^{k-1})$ 
    for $k\in\N$ as $n \to \infty$, we obtain    \eqref{Eq:difference_of_moments}.
\end{proof}

\begin{pr}\label{Prop:difference_of_dis_ftns}
    Let $F_{\alpha,x}^{(n)}$ and $G_{\alpha,x}^{(n)}$ be the distribution functions of $\mu_{\alpha,x}^{(n)}$ and $\nu_{\alpha,x}^{(n)}$, respectively. Then, there exists a constant $C>0$ such that
        \[
            \sup_{z\in\R}\left|F_{\alpha,x}^{(n)}(z)-G_{\alpha,x}^{(n)}(z)\right|
            \le \frac{C}{\sqrt{n}},
            \qquad n\in\N.
        \]
\end{pr}

\begin{proof}
Recall that $\widetilde{F}_{\alpha, x}^{(n)}(z)$ denotes the distribution function of $\widetilde{\mu}_{\alpha, x}^{(n)}$
that is the law of $\widetilde{S}_{\alpha, x}^{(n)}$ (see \eqref{Eq:normalized-mu}). 
We have
    \begin{align}
        \sup_{z\in\R}\left|F_{\alpha,x}^{(n)}(z)
        -G_{\alpha,x}^{(n)}(z)\right| 
        &=\sup_{z\in\R}\left|\widetilde F_{\alpha,x}^{(n)}(z)-G_{\alpha,x}^{(n)}\left(\E[S_{\alpha,x}^{(n)}]+ \sqrt{\Var(S_{\alpha,x}^{(n)})}z\right)\right| \nn \\
        &=\sup_{z\in\R}\left|\widetilde F_{\alpha,x}^{(n)}(z)-\Phi(z)\right|+\sup_{z\in\R}\left|G_{\alpha,x}^{(n)}
        \left(\E[S_{\alpha,x}^{(n)}]
        +\sqrt{\Var(S_{\alpha,x}^{(n)})}z\right)-\Phi(z)\right|.
        \label{Eq:F-G-estimate-1}
    \end{align}
The first term on the right-hand side of \eqref{Eq:F-G-estimate-1}
is $O(1/\sqrt{n})$ 
from Theorem~\ref{Thm:Berry-Esseen}.
The second term on the right-hand side of \eqref{Eq:F-G-estimate-1} is equal to
    \begin{align}
        &\sup_{z'\in\R}\left|G_{\alpha,x}^{(n)}\left(\lfloor\alpha n\rfloor x/\alpha+ \sqrt{\lfloor\alpha n\rfloor x(1-x)/\alpha^2}z'\right)\right.\nonumber\\
        &\phantom{\sup_{z'\in\R}}-\left.\Phi\left(\frac{\lfloor\alpha n\rfloor x/\alpha-\E[S_{\alpha,x}^{(n)}]}{\sqrt{\Var( S_{\alpha,x}^{(n)})}}+\sqrt{\frac{\lfloor\alpha n\rfloor x(1-x)/\alpha^2}{\Var(S_{\alpha,x}^{(n)})}}z'\right)\right|\nonumber\\
        &\le \sup_{z\in\R}\left|G_{\alpha,x}^{(n)}\left(\lfloor\alpha n\rfloor x/\alpha+ \sqrt{\lfloor\alpha n\rfloor x(1-x)/\alpha^2}z\right)-\Phi(z)\right|\nonumber\\
        &\quad+\sup_{z\in\R}\left|\Phi(z)-\Phi\left(\frac{\lfloor\alpha n\rfloor x/\alpha-\E[S_{\alpha,x}^{(n)}]}{\sqrt{\Var(S_{\alpha,x}^{(n)})}}+\sqrt{\frac{\lfloor\alpha n\rfloor x(1-x)/\alpha^2}{\Var(S_{\alpha,x}^{(n)})}}z\right)\right|.
        \label{Eq:F-G}
    \end{align}
The first term on the right-hand side of \eqref{Eq:F-G} is $O(1/\sqrt{n})$ from the classical Berry--Esseen type estimate. 
We give an upper bound of the second term on the right-hand side 
of \eqref{Eq:F-G}.
We write 
\[
\ve_n=\frac{\lfloor\alpha n\rfloor x/\alpha-\E[S_{\alpha,x}^{(n)}]}{\sqrt{\Var(S_{\alpha,x}^{(n)})}}
\quad\text{and}\quad
s_n=\sqrt{\frac{\lfloor\alpha n\rfloor x(1-x)/\alpha^2}{\Var(S_{\alpha,x}^{(n)})}}-1.
\]
Then, $\ve_n=O(1/\sqrt{n})$ and $s_n=O(1/\sqrt{n})$ as $n\to\infty$.
The second term on the right-hand side of \eqref{Eq:F-G} is dominated by
\begin{equation}\label{Eq:F-G2}
\sup_{z\in\R}\left|\Phi(z)-\Phi\left((1+s_n) z\right)\right|
+\sup_{z \in\R}\left|\Phi((1+s_n)z)-\Phi(\ve_n+(1+s_n)z)\right|.
\end{equation}
The second term of \eqref{Eq:F-G2} is $O(\ve_n)=O(1/\sqrt{n})$ from the Lipschitz continuity of $\Phi$.
Concerning the first term, we may assume $|s_n|\le 1/2$.
Then, 
\begin{align*}
|\Phi(z)-\Phi((1+s_n) z)|
&\le |s_n z|\sup_{w\in[z \wedge(1+s_n)z, z\vee(1+s_n)z]}\Phi'(w)\\
&\le \frac{|s_n|}{1\wedge(1+s_n)}\sup_{w\in\R}|w|\Phi'(w)\\
&\le 2|s_n|\sup_{w\in\R}|w|\Phi'(w).
\end{align*}
Since $|w|\Phi'(w)$ is bounded on $\R$, we obtain that 
    \[
    \sup_{z\in\R}|\Phi(z)-\Phi((1+s_n)z)|=O(s_n)=O\left(\frac{1}{ \sqrt{n}}\right)
    \]
$n \to \infty$. 
Thus, we have the desired estimate.
\end{proof}

\begin{acks}[Acknowledgments]
The first-named author was supported by 
JSPS KAKENHI Grant Numbers 19H00643 and 25K07056. 
The second-named author was supported by
JSPS KAKENHI Grant Number 23K12986.
\end{acks}

\end{document}